\documentclass[12pt]{amsart}

\RequirePackage{mathtools}
\RequirePackage{amsopn}
\RequirePackage{amsfonts}
\RequirePackage{paralist}
\RequirePackage{amssymb}
\RequirePackage{amsthm}
\RequirePackage{mathrsfs}
\usepackage{tikz}
\usetikzlibrary{cd,decorations.pathmorphing,decorations.markings}
\usepackage[unicode,psdextra]{hyperref}
\usepackage{cleveref} %% not compatible with showonlyrefs
\usepackage{enumitem}  % to adjust enumerate env

\theoremstyle{plain}
\newtheorem{thm}{Theorem}[section]

\newtheorem{cor}[thm]{Corollary}
\newtheorem{lemma}[thm]{Lemma}
\newtheorem{prop}[thm]{Proposition}

\theoremstyle{definition}
\newtheorem{definition}[thm]{Definition}

\theoremstyle{remark}
\newtheorem{remark}[thm]{Remark}
\newtheorem{example}[thm]{Example}

\numberwithin{equation}{section}
\evensidemargin=\oddsidemargin
\DeclareMathSymbol{\rtimes}{\mathbin}{AMSb}{"6F}

\def\ibind#1{\mathop{#1\mathord{\mathop{\text{--}}}}\!\Ind\nolimits}

\newcommand\T{\mathbf{T}}

\newcommand\set[1]{\{\,#1\,\}}
\newcommand\sset[1]{\{#1\}}

\def\<{\langle}
\def\>{\rangle}
\let\ipscriptstyle=\scriptscriptstyle
\def\lipsqueeze{{\mskip -3.0mu}}
\def\ripsqueeze{{\mskip -3.0mu}}
\def\ipcomma{\nobreak\mathrel{,}\nobreak}
\newbox\ipstrutbox
\setbox\ipstrutbox=\hbox{\vrule height8.5pt% depth 3.5pt
width 0pt}
\def\ipstrut{\copy\ipstrutbox}
\def\lip#1<#2,#3>{\mathopen{\relax_{\ipstrut\ipscriptstyle{
#1}}\lipsqueeze
\langle} #2\ipcomma #3 \rangle}
\def\blip#1<#2,#3>{\mathopen{\relax_{\ipstrut
\ipscriptstyle{ #1}}\lipsqueeze\bigl\langle} #2\ipcomma #3 \bigr\rangle}
\def\rip#1<#2,#3>{\langle #2\ipcomma #3
\rangle_{\ripsqueeze\ipstrut\ipscriptstyle{#1}}}
\def\brip#1<#2,#3>{\bigl\langle #2\ipcomma #3
\bigr\rangle_{\ripsqueeze\ipstrut\ipscriptstyle{#1}}}
\def\angsqueeze{\mskip -6mu}
\def\smangsqueeze{\mskip -3.7mu}
\def\trip#1<#2,#3>{\langle\smangsqueeze\langle #2\ipcomma #3
\rangle\smangsqueeze\rangle_{\ripsqueeze\ipstrut\ipscriptstyle{#1}}}
\def\btrip#1<#2,#3>{\bigl\langle\angsqueeze\bigl\langle #2\ipcomma
#3
\bigr\rangle
\angsqueeze\bigr\rangle_{\ripsqueeze\ipstrut\ipscriptstyle{#1}}}
\def\tlip#1<#2,#3>{\mathopen{\relax_{\ipstrut\ipscriptstyle{
#1}}\lipsqueeze \langle\smangsqueeze\langle} #2\ipcomma #3
\rangle\smangsqueeze\rangle}
\def\btlip#1<#2,#3>{\mathopen{\relax_{\ipstrut\ipscriptstyle{
#1}}\lipsqueeze
\bigl\langle\angsqueeze\bigl\langle} #2\ipcomma #3
\bigr\rangle\angsqueeze\bigr\rangle}

\def\ip(#1|#2){(#1\mid #2)}
\def\bip(#1|#2){\bigl(#1 \mid #2\bigr)}
\def\Bip(#1|#2){\Bigl( #1 \bigm| #2 \Bigr)}
\usepackage[backend=biber,
style=numeric,
citestyle=numeric-comp,
maxbibnames=10]{biblatex} %% Simulates alpha
\DeclareFieldFormat[article]{volume}{\mkbibbold{#1}}
\DeclareFieldFormat[article]{title}{#1}

\renewbibmacro*{in:}{%
  \relax}

\AtEveryBibitem{%
  \clearfield{mrclass}%
  \clearfield{mrnumber}%
  \clearfield{mrreviewer}%
  \clearfield{doi}%
  \clearfield{number}%
  \clearfield{issn}%
  \clearfield{isbn}%
  \ifentrytype{online}{}{% Remove url except for online sources
    \clearfield{url} } }

\usepackage{mathscinet}

\makeatletter
\def\labelenumi{\textnormal{(\@alph\c@enumi)}}
\def\theenumi{\@alph \c@enumi}
\def\labelenumii{\textnormal{(\@roman\c@enumii)}}
\def\theenumii{\@roman \c@enumii}
\newcount\charno
\def\alphapart#1{\charno=96
\advance\charno by#1\char\charno}

\makeatother
\newcommand\go{G\z}

\newcommand\Ind{\operatorname{Ind}}

\newcommand\Prim{\operatorname{Prim}}

\newcommand\supp{\operatorname{supp}}
\newcommand\Iso{\operatorname{Iso}}

\newcommand\cs{\ensuremath{\Cst}}  %% For compatibility with linear-v03
\newcommand\Cst{C^{*}}
\newcommand\red{r}

\newcommand\z{^{(0)}}
\newcommand\inv{^{-1}}
\newcommand\comp{^{(2)}}

\newcommand\dd{d}

\newcommand\gmiso{G/\Iso(G)}
\newcommand\alex{\operatorname{\mathbb{A}}}
\newcommand\cG{\alex(G)}

\newcommand\I{\mathcal{I}}
\newcommand\Ig{\I_{G}}
\renewcommand\O{\mathcal{O}}

\newcommand\gugo{G\backslash\go}
\newcommand\gmgo{\gugo} %% If this is correct, replace these

\crefformat{enumi}{#2\textup{(#1)}#3}

\newcommand{\NN}{\mathbf{N}}

\newcommand{\ZZ}{\mathbf{Z}}

\newcommand{\TT}{\mathbf{T}}

\newcommand\god[1]{\go_{#1}}
\newcommand\godp[1]{\go_{#1+}}
\newcommand\godm[1]{\go_{#1-}}
\newcommand\card[1]{\bigl| #1 \bigr|}

\newcommand{\sign}{{\mathrm{sign}}}
\DeclareMathOperator{\range}{range}

\newcommand{\plusone}{^{+1}}
\newcommand{\nuc}{{\mathrm{nuc}}}
\newcommand{\Ff}{\mathcal{F}}

\DeclareMathOperator{\dad}{\textsc{dad}}
\newcommand{\lsp}{\operatorname{span}}

\newcommand\Indx[1]{\Ind_{G(#1)}^{G}}
\newcommand\lambdax[1]{\lambda_{G(#1)}}
\newcommand\so{\Sigma_{0}}

\begin{document}
\begin{abstract} We characterise when the $\Cst$-algebra $\Cst(G)$ of
  a locally compact and Hausdorff groupoid $G$ is subhomogeneous, that
  is, when its irreducible representations have bounded finite
  dimension; if so we establish a bound for its nuclear dimension in
  terms of the topological dimensions of the unit space of the
  groupoid and the spectra of the primitive ideal spaces of the
  isotropy subgroups.  For an \'etale groupoid $G$, we also establish
  a bound on the nuclear dimension of its $\Cst$-algebra provided the
  quotient of $G$ by its isotropy subgroupid has finite dynamic
  asymptotic dimension in the sense of Guentner, Willet and Yu.  Our
  results generalise those of C.~B\"oncicke and K.~Li to groupoids
  with large isotropy, including graph groupoids of directed
  graphs. We find that all graph $\Cst$-algebras that are stably
  finite have nuclear dimension at most $1$.  We also show that the
  nuclear dimension of the $\Cst$-algebra of a twist over $G$ has the
  same bound on the nuclear dimension as for $\Cst(G)$ and the twisted
  groupoid $\Cst$-algebra.
\end{abstract}

\title[Nuclear dimension of  groupoid
  $\Cst$-algebras]{\boldmath Nuclear dimension of  groupoid
  $\Cst$-algebras with large abelian isotropy, with applications to
  $\Cst$-algebras of directed graphs and twists} 

\author[an Huef]{Astrid an Huef}
\address{School of Mathematics and Statistics \\ Victoria University
  of Wellington \\ P.O.\@ Box 600 \\ Wellington 6140 NEW ZEALAND}
\email{astrid.anhuef@vuw.ac.nz}

\author[Williams]{Dana P. Williams}
\address{Department of Mathematics\\ Dartmouth College \\ Hanover, NH
  03755-3551 USA}
\email{dana.williams@dartmouth.edu}



\thanks{This research was supported by Marsden grant 21-VUW-156 of the
  Royal Society of New Zealand and the Shapiro Fund of Dartmouth
  College. We thank Gwion Evans, Alex Kumjian, Sergey Neshveyev, Aidan Sims and 
  Danie van Wyk for helpful discussions. }  \date{19 February 2025, with revisions 24 August 2025}

\maketitle

\section{Introduction}
\label{sec:introduction}
Nuclear dimension of $\Cst$-algebras, defined by Winter and Zacharias
in \cite{WinterZacharias:nuclear}, is a non-commutative generalisation
of topological covering dimension.  For example, the nuclear dimension
of a $\Cst$-algebra with continuous trace is the topological dimension
of its spectrum. Furthermore, the nuclear dimension of a
subhomogeneous $\Cst$-algebra is the maximum of the topological
dimensions of the spectra of its maximal subquotients with continuous
trace. Nuclear dimension of $\Cst$-algebras and subhomogeneous
$\Cst$-algebras are important ingredients in the classification
programme for $\Cst$-algebras.  The current state-of-the-art of the
programme is \cite[Corollary~D]{Tikuisis-White-Winter:QD}: every
unital, separable, simple, infinite-dimensional $\Cst$-algebra with
finite nuclear dimension that satisfies the hypotheses of a certain
Universal Coefficient Theorem (UCT) is classifiable by its Elliott
invariant. Moreover, if the $\Cst$-algebra is stably finite, then it
is an inductive limit of subhomogeneous $\Cst$-algebras (see, for
example, \cite[\S1]{Li:Classif-Cartan}).  \medskip

Deciding if the nuclear dimension of a particular (possibly non-simple
or non-unital) $\Cst$-algebra is finite, and computing or bounding its
nuclear dimension, are all very challenging problems \cite{W+20,
  BGSW22, EGM19, EM18, Bonicke-Li-Nuclear-dim, Faurot-Schafhauser,
  RuizSimsTomforde:nucleardim}.  For example, it is currently not
known if simple $\Cst$-algebras of $2$-graphs, generally viewed as a
tractable class, have finite nuclear dimension. In
\cite[\S6]{EvansSims:AF}, Evans and Sims present two $2$-graphs
$\Lambda_I$ and $\Lambda_{II}$ whose $\Cst$-algebras are simple, and,
respectively, known to be approximately finite-dimensional (AF) and
AF-embeddable, and have the same Elliott invariant.  Since
$\Cst$-algebras of $k$-graphs satisfy the UCT, if $\Cst(\Lambda_I)$
and $\Cst(\Lambda_{II})$ have finite nuclear dimensions, then they are
isomorphic.

\medskip Even for the $\Cst$-algebra of a directed graph $E$ we
currently only have bounds on their nuclear dimension in special
cases.  A simple graph $\Cst$-algebra is either AF or purely infinite
by \cite{KumjianPaskRaeburn-graphs}, and hence has nuclear dimension
$0$ by \cite{WinterZacharias:nuclear} or $1$ by
\cite{RuizSimsSorensen}.  Faurot and Schafhauser observed in
\cite{Faurot-Schafhauser} that a $\Cst$-algebra of a finite graph has
finite nuclear dimension; by reducing to finite graphs, they show that
the $\Cst$-algebra of a graph satisfying condition (K), where every
return path has an entrance, has nuclear dimension at most $2$.  See
also \cite{Evington-Ng-Sims-White} for the nuclear dimension of a
$\Cst$-algebra of a finite graph, and more generally, the nuclear
dimension of an extension.  In \cite{RuizSimsTomforde:nucleardim},
Ruiz, Sims and Tomforde consider $\Cst(E)$ with a purely infinite
ideal $I$ with only finitely many ideals such that $\Cst(E)/I$ is
approximately finite dimensional, and show that it has nuclear
dimension at most $2$.

\medskip

As an application of our results about $\Cst$-algebras of groupoids
with large isotropy subgroups, we consider the $\Cst$-algebra of a
directed graph where no return path has an entrance. By
\cite{Schafhauser-AFE}, this is precisely the class of graph
$\Cst$-algebras that are stably finite (equivalently, AF-embeddable),
and we conclude that they all have nuclear dimension at most $1$.  Our
results are distinct from those in \cite{Faurot-Schafhauser,
  RuizSimsTomforde:nucleardim, Evington-Ng-Sims-White} and provide
significant evidence towards a positive answer to
\cite[Question~C]{Evington-Ng-Sims-White} which asks if all graph
$\Cst$-algebras have nuclear dimension at most $1$.  Our techniques
use a quotient of the graph groupoid $G_E$ of $E$ by its isotropy
subgroupoid, and this quotient groupoid is tractable because the
isotropy subgroupoid is open when no return path in $E$ has an entry.
In general, the isotropy subgroupoid is not open, and this restricts
the class of graphs to which our results can apply (see  \Cref{can't
   deal}).
\medskip

In \cite{GWY:2017:DAD}, Guentner, Willet and Yu developed a notion of
dynamic asymptotic dimension for actions of discrete groups on spaces,
and, more generally, for \'etale groupoids.  Roughly speaking, a
groupoid $G$ has dynamic asymptotic dimension at most $d\in\NN$, if
for every open, precompact and symmetric subset $K$ of $G$ there
exists a partition $U_0, U_1, \dots, U_d$ of the range of $K$ in the
unit space such that the subgroupoids $H_i$ generated by elements of
$K$ with range and source in one of the $U_i$, are open and precompact
in $G$. Consequently, each $\Cst(H_i)$ is a subhomgeneous
$\Cst$-algebra that can be viewed as a $\Cst$-subalgebra of $\Cst(G)$.
Theorem~8.6 of \cite{GWY:2017:DAD} states that if $G$ is a principal,
\'etale groupoid with dynamic asymptotic dimension $d$, then the
nuclear dimension of $\Cst(G)$ is bounded by a number depending on $d$
and the topological dimension of unit space of $G$. The same bound was
subsequently found for twisted groupoid $\Cst$-algebras $\Cst(E;G)$
for twists $E$ over an \'etale groupoid $G$, first when $G$ is
principal \cite[Theorem~4.1]{CDaHGV} and then for non-principal $G$
\cite[Theorem~3.2]{Bonicke-Li-Nuclear-dim}.  However, if $G$ has
finite dynamic asymptotic dimension, then the isotropy subgroups of
$G$ must be locally finite, that is, their finitely generated
subgroups must be finite.

\medskip

For a directed graph $E$, the dynamic asymptotic dimension of the
graph groupoid $G_E$ is either $0$ or $\infty$, and hence is a poor
predictor of the nuclear dimension of $\Cst(E)$. The problem is that
the isotropy subgroups are either trivial or isomorphic to $\ZZ$.
Nevertheless, \cite[Corollary~5.5]{CDaHGV}, based on results from
\cite{Muhly-Williams-Renault:tr3, Clark-anHuef-RepTh}, already showed
that when the orbit space of a groupoid is $T_1$, and the isotropy
subgroups are abelian and vary continuously, then looking at the
quotient groupoid of $G$ by its isotropy subgroupoid yields good
results.

\medskip 

We have developed much of our theory for possible non-\'etale
locally compact, Hausdorff groupoids.  In particular, in
\S\ref{sec:quasi-orbit-map} we show that the quasi-orbit map
associated to a groupoid $G$ is continuous, extending
\cite[Proposition~2.9]{Bonicke-Li-Nuclear-dim} from \'etale to general
locally compact, Hausdorff groupoids. We also show that if $G$ is amenable,  then the quasi-orbit map is
open.  Furthermore, in \S\ref{sec:subh-cs-algebr} we characterise when
$\Cst(G)$ is subhomogeneous, and if so, find a bound on its nuclear
dimension, again extending results for \'etale groupoids from
\cite{Bonicke-Li-Nuclear-dim}.  We also observe that if $\Cst(G)$ is
subhomogeneous, then $G$ is amenable. These results do not require the
isotropy subgroups to vary continuously.  

\medskip

\Cref{main thm} establishes a bound on the nuclear dimension of
$\Cst(G)$ for an \'etale groupoid $G$ with continuously varying
isotropy subgroups that are subhomogeneous.  The bound depends on the
topological dimensions of the unit space and the spectra of the
isotropy subgroups of $G$, and the dynamic asymptotic dimension of the
quotient groupoid of $G$ by its isotropy subgroupoid. To even ask that
the quotient groupoid has finite dynamic asymptotic dimension we need
its topology to be locally compact.  Theorem~\ref{main thm} yields our
application to graph algebras in \S\ref{sec directed graphs}, already
discussed above, which is based on a graph-theoretic construction that
may be of independent interest. Given a directed graph $E$ in which no
return path has an entrance, we construct a graph $F$ with no return
paths such that there is an isomorphism of topological groupoids from
the quotient of the graph groupoid $G_E$ by its isotropy subgroups
onto an open subgroupoid of $G_F$.  Since the dynamic asymptototic
dimension of $G_F$ is $0$, this shows that the quotient of $G_E$ has
dynamic asymptotic dimension $0$, and \Cref{main thm} applies, whence  the nuclear dimension of $\Cst(E)$ is at most $1$.
\medskip

We then consider twists and twisted groupoid $\Cst$-algebras.
\Cref{thm-twist-and-recover} establishes a bound on the nuclear
dimension of the $\Cst$-algebra $\Cst(\Sigma)$ of a twist $\Sigma$
over an \'etale groupoid $G$.  Since $\Sigma$ is an extension of $G$
by a trivial circle bundle, it is never \'etale. Nevertheless, we were
able to stretch our techniques to show directly that the nuclear
dimension of $\Cst(\Sigma)$ is bounded by the topological dimension of
the unit space and the dynamic asymptotic dimension of $G$.  Since the
twisted groupoid $\Cst$-algebra is a direct summand of $\Cst(\Sigma)$
by \cite{Brown-anHuef, IKRSW} we obtain the same bound on its
nuclear dimension as a corollary, recovering the main theorem
\cite[Theorem~3.2]{Bonicke-Li-Nuclear-dim}. Of course we could have
obtained \Cref{thm-twist-and-recover} by applying
\cite[Theorem~3.2]{Bonicke-Li-Nuclear-dim}, but we wanted to test our
techniques in the non-\'etale setting as well as advocate for the
philosophy that to gain information about twisted groupoid
$\Cst$-algebras looking at the $\Cst$-algebra of the twist can be
invaluable. Here, for example, we can deduce the bound on the nuclear
dimension of twisted groupoid $\Cst$-algebras without ever analysing
their subhomogeneous subalgebras nor twisted group $\Cst$-algebras.

\section{Preliminaries}
\label{sec:preliminaries}

Throughout, $G$ is a second countable, locally compact and
Hausdorff groupoid with unit space $G\z$ equipped with a left Haar
system $\set{\lambda^{u}}_{u\in\go}$.  Our primary reference for
groupoids and their \cs-algebras is \cite{Williams:groupoid}.  Since
$G$ has a Haar system, the range and source maps $r,s\colon G\to G\z$,
given by $r(\gamma)=\gamma\gamma\inv$ and
$s(\gamma)=\gamma\inv \gamma$, are open as well as continuous. The set
of composable pairs $\set{(\beta,\gamma) : s(\beta)=r(\gamma)}$ is
denoted by $G\comp$.  For subsets $A$ and $B$ of $G$, we set
$AB\coloneqq\set{\alpha\beta:\text{$\alpha\in A$, $\beta\in B$, and
    $(\alpha,\beta)\in G\comp$}}$.  In particular, if $u\in\go$, then
we write $G_{u}=Gu=\set{\gamma\in G:s(\gamma)=u}$, and similarly for
$G^{u}=uG$.  The \emph{isotropy subgroup} at $u$ is
$G(u)=G^{u}\cap G_{u}$, and
\[\Iso(G)\coloneqq\set{\gamma\in G :r(\gamma)=s(\gamma)}\] is the
\emph{isotropy subgroupoid} of $G$.  Although $\Iso(G)$ is always a
closed subgroupoid of $G$, it will have a Haar system if and only if
the restriction of $r$, or $s$, to $\Iso(G)$ remains open---see
\cite[Theorem~6.12]{Williams:groupoid}.  Note that $G$ acts on the
left of $\go$ via $\gamma\cdot s(\gamma)=r(\gamma)$
\cite[Example~2.6]{Williams:groupoid}, and we write $[u]$ for the
\emph{orbit} $G\cdot u$ of $u$ so that $[u] = r(s^{-1}(\{u\}))$.

We will often require our groupoid $G$ to have the additional property
that the range and source maps are local homeomorphisms; such a
groupoid is called \emph{\'etale}, and we can equip it with a Haar
system consisting of counting measure on each fibre $G^{u}$. In
particular, in this paper ``\'etale'' includes second countable,
locally compact, Hausdorff, with Haar system of counting measures.

When working with open maps it is often useful to observe that we can
lift convergent sequences in the range.  We will use the following
lemma which allows a sharpening of Fell's Criterion
\cite[Proposition~1.15]{Williams:Crossed} that avoids passing to
subsequences in the separable case.  We learned this trick from
\cite[Proposition~2.4]{siw:gta71}.

\begin{lemma}[Sequence Lifting] \label{lem-seq-lift} Suppose that $X$
  and $Y$ are topological spaces with $X$ first countable.  Let
  $f\colon X\to Y$ be an open surjection.  Let $\{y_{n}\}\subset Y$ be
  a \emph{sequence} converging to $f(x)$.  Then there is a sequence
  $\{x_n\}\subset X$ such that $x_{n}\to x$ and such that
  $f(x_{n})=y_{n}$.
\end{lemma}

\begin{proof}
  Let $\set{U_{n}:n\ge1}$ be a countable neighborhood base at $x$ such
  that $U_{n+1} \subset U_{n}$ for all $n\ge1$.  Also let
  $F_{n}=f^{-1}(y_{n})$ for $n\ge 1$.  Since $f$ is open, $f(U_{n})$
  is a neighborhood of $f(x)$, and there exists $m(i)$ such that
  $n\ge m(i)$ implies that $F_{n}$ meets $U_{i}$.  We can arrange that
  $m(i)<m(i+1)$ for all $i$.  Let $x_{n}\in F_{n}$ be arbitrary if
  $n<m(1)$.  Since $F_{n}\cap U_{i}\not=\emptyset$ if $n\ge m(i)$, we
  can choose $x_{n}\in F_{n}\cap U_{i}$ if $m(i)\le n <m(i+1)$.  If
  $U$ is any neighborhood of $x$, then there is an $i_{0}$ such that
  $x\in U_{i_{0}} \subset U$.  Then if $n\ge m(i_{0})$, we have
  $x_{n}\in U$.  Therefore $x_{n}\to x$ in $U$.
\end{proof}

We equip the space $\mathcal{C}(G)$ of closed subsets of $G$ with the
Fell topology from \cite{Fell-1962}. Then the \emph{isotropy subgroups
  of $G$ vary continuously} if the function $u\mapsto G(u)$ from $G\z$
to the subspace of $\mathcal{C}(G)$ of closed subgroups of $G$ is
continuous \cite[\S H.4]{Williams:groupoid}.

Since $\Iso(G)$ is a closed subgroupoid of $G$, $G$ is a free and
proper right $\Iso(G)$-space.  Furthermore, the orbit space $\gmiso$
is a principal groupoid such that the quotient map
$\rho\colon G\to \gmiso$ is a homomorphism (see, for example,
\cite[Lemma~2.2]{IKRSW}).  Since $\Iso(G)$ may not have open range and
source maps, the quotient topology on $\gmiso$ may not even be locally
compact nor Hausdorff.  Furthermore, unless $\rho$ is open, it is not
clear that $G/\Iso(G)$ is a topological groupoid.

\begin{lemma}\label{lem hypotheses for lc quotient}
  Let $G$ be an \'etale groupoid and let $\rho\colon G\to G/\Iso(G)$
  be the quotient map.  Then the following are equivalent:
  \begin{enumerate}
  \item\label{lem hypotheses for lc quotient 1} $\Iso(G)$ is open in
    $G$;
  \item\label{lem hypotheses for lc quotient 2} $u\mapsto G(u)$ is
    continuous;
  \item\label{lem hypotheses for lc quotient 3}
    $\bigl(G/\Iso(G)\bigr)\z$ is open in $G/\Iso(G)$;
  \item\label{lem hypotheses for lc quotient 4} $\rho$ is open.
  \end{enumerate} If the equivalent conditions \cref{lem hypotheses
    for lc quotient 1}--\cref{lem hypotheses for lc quotient 4} hold,
  then $\gmiso$ is an \'etale groupoid.
\end{lemma}

\begin{proof}
  The equivalence of (\ref{lem hypotheses for lc quotient 1}) and
  (\ref{lem hypotheses for lc quotient 2}) is
  \cite[Lemma~5.1]{CDaHGV}.  The equivalence of (\ref{lem hypotheses
    for lc quotient 1}) and (\ref{lem hypotheses for lc quotient 3})
  follows because
  $\rho^{-1}\bigl( \bigl( G/\Iso(G)\bigr)\z\bigr)=\Iso(G)$ is open by
  the definition of the quotient topology.  That (\ref{lem hypotheses
    for lc quotient 2})$\Longleftrightarrow$(\ref{lem hypotheses for
    lc quotient 4}) holds follows from
  \cite[Theorem~6.12]{Williams:groupoid} and
  \cite[Ex~6.3.2(b)]{Williams:groupoid} (and does not require $G$ to
  be \'etale).

  If (\ref{lem hypotheses for lc quotient 1})--(\ref{lem hypotheses
    for lc quotient 4}) hold, then $\gmiso$ is locally compact
  Hausdorff by \cite[Proposition~2.18]{Williams:groupoid}---we need
  (\ref{lem hypotheses for lc quotient 2}) and
  \cite[Theorem~6.12]{Williams:groupoid} to see that $\Iso(G)$ has a
  Haar system in order to apply
  \cite[Proposition~2.18]{Williams:groupoid}.

  It still remains to check that $\gmiso$ is a topological
  groupoid---that is, that groupoid operations are continuous.  But
  this follows since $\rho$ is open using Lemma~\ref{lem-seq-lift}.
  For example, if
  $\bigl( \rho(\gamma_{n}),\rho(\eta_{n}) \bigr)\to
  \bigl(\rho(\gamma),\rho(\eta)\bigr)$ in $(\gmiso)\comp$, then we can
  use Lemma~\ref{lem-seq-lift} to assume that $\gamma_{n}\to \gamma$
  and $\eta_{n}\to \eta$.  Then $\gamma_{n}\eta_{n}\to \gamma\eta$ and
  $\rho(\gamma_{n})\rho(\eta_{n})\to \rho(\gamma)\rho(\eta)$ since
  $\rho$ is a continuous homomorphism.
\end{proof}

Examples where $G$ is \'etale but the quotient groupoid $G/\Iso(G)$ is
not \'etale abound:

\begin{example}\label{can't deal}
  Let $E$ be a row-finite directed graph that is cofinal and
  aperiodic. Let $G_E$ be the associated graph groupoid (see \Cref{sec
    directed graphs}), which is \'etale.  Suppose that $E$ has a
  periodic infinite path $x$. By aperiodicity there exists a sequence
  $\{x_n\}$ of aperiodic inifnite paths such that $x_n\to x$. Then
  $\{\{x_n\}\}=\{G_E (x_n)\}$ cannot converge to $G_E( x)\neq\{x\}$.
  Since the isotropy subgroups do not vary continuously, by \Cref{lem
    hypotheses for lc quotient} the unit space of the quotient
  groupoid $G_E/\Iso(G_E)$ is not open, and hence $G_E/\Iso(G_E)$ is
  not \'etale.  So our techniques do not apply in this situation.  It
  is not even clear that $G_E/\Iso(G_E)$ is a locally compact
  groupoid. Similarly, our techniques do not apply to the
  $\Cst$-algebras of the $2$-graphs $\Lambda_I$ and $\Lambda_{II}$
  mentioned in the introduction.
\end{example}

Let $G$ be a locally compact, Hausdorff and \'etale groupoid with
non-compact unit space $G\z $, and let $G\z \cup \{\infty\}$ be the
one-point compactification of $G\z $. Then
\[\cG\coloneqq G\cup \{\infty\}\] is a locally compact, Hausdorff,
\'etale groupoid, with compact unit space $G\z\cup\{\infty\}$, with
the following structure:
$\cG\comp \coloneqq G\comp \cup\{(\infty, \infty)\}$, multiplication,
inversion, and $r$ and $s$ are extended from $G$ to $\cG$ by setting
$r\inv(\infty)=s\inv(\infty)=\{\infty\}$; the set consisting of $\cG$,
all open sets in $G\z\cup\{\infty\}$, and all open sets in $G$ is a
basis for a topology on $\cG$.  Then $\cG$ is called the
\emph{Alexandrov groupoid} and is studied in detail in
\cite[\S3]{CDaHGV}.  It is straightforward to verify, using the two
criteria of \cite[Lemma~H.2]{Williams:groupoid}, that if the isotropy
subgroups vary continuously on $G\z$, then they also vary continuously
on $\cG\z$.

In Section~\ref{sec:nuclear-dimension-cs}, we will need to consider
the Aleaxandrov groupoid of the quotient groupoid $\gmiso$.  The
following will be useful.

\begin{lemma}\label{dad quotient}
  Let $G$ be an \'etale groupoid and let $\rho\colon G\to G/\Iso(G)$
  be the quotient map. Suppose that the isotropy subgroups vary
  continuously.  Write
  \[\cG=G\cup\{\infty\}\quad \text{and}\quad \alex(G/\Iso(G))
    =G/\Iso(G)\cup\{{\infty'}\}
  \]
  for the Alexandrov groupoids of $G$ and $G/\Iso(G)$,
  respectively. Define $\bar\rho\colon \cG\to \alex(G/\Iso(G))$ by
  $\bar{\rho}(\infty)={\infty'}$ and $\bar{\rho}(\gamma)=\rho(\gamma)$
  for $\gamma\in G$. Then $\bar\rho$ is a continuous, open, surjective
  homomorphism that factors through an isomorphism of $\cG/\Iso(\cG)$
  onto $\alex(G/\Iso(G))$.
\end{lemma}

\begin{proof} The quotient map $\rho$ is a continuous, surjective
  homomorphism. The isotropy subgroups vary continuously, and so
  $\rho$ is also open by \Cref{lem hypotheses for lc quotient}. In
  particular, the restriction of $\rho$ to the unit space is a
  homeomorphism of $G\z$ onto $(G/\Iso(G))\z$.

  Since $(\alpha,\beta)\in \alex(G)\comp$ if and only if
  $(\alpha,\beta)\in G\comp$ or $\alpha=\beta=\infty$, it follows that
  $\bar\rho$ is a surjective homomorphism.

  To see that $\bar\rho$ is continuous and open, it suffices to
  consider nontrivial basic open sets containing ${\infty'}$ and
  $\infty$, respectively.  To this end, let $K, L$ be compact subsets
  of $(G/\Iso(G))\z$ and $G\z$, respectively. Then since $\rho$
  restricted to the unit space is a homeomorphism,
  \[
    \bar{\rho}\inv \big((G/\Iso(G))\z\setminus
    K\cup\{{\infty'}\}\big)=\rho^{-1}\bigl( G\z\setminus
    \rho\inv(K)\cup\{\infty\} \bigr)
  \]
  is open in $\cG$ and
  \[
    \bar\rho\big(G\z\setminus
    L\cup\{\infty\}\big)=(G/\Iso(G))\z\setminus \rho(L)\cup\{{\infty'}\}
  \]
  is open in $\alex( G/\Iso(G))$.  Thus $\bar\rho$ is continuous and
  open.

  Next, let $\alpha, \beta\in \cG$; we claim  that  $\bar\rho(\alpha)=\bar\rho(\beta)$ if and only if  there exists $\gamma\in \Iso(\cG)$ such that   $\alpha=\beta\gamma$. First, suppose that $\bar\rho(\alpha)=\bar\rho(\beta)$. If $\alpha=\infty$, then $\beta=\infty$, and we take  $\gamma=\infty \in \Iso(\cG)$ to get $\alpha=\beta\gamma$. Second, suppose there exists $\gamma\in \Iso(\cG)$ such that $\alpha=\beta\gamma$. Since $\bar\rho$ is a homomorphism,  $\bar\rho(\gamma)$ is a unit, and  $\bar\rho(a)=\bar\rho(\beta)\bar\rho(\gamma)=\bar\rho(\beta)$.
 Thus $\bar\rho$ factors through an isomorphism of $\cG/\Iso(\cG)$ onto $\alex( G/\Iso(G))$ as claimed.
\end{proof}

The definitions of the full and reduced $\Cst$-algebras of $G$ are
given in detail in, for example, \cite[\S1.4]{Williams:groupoid}.  We
record some of the basics here for convenience.  We assume that our
groupoids $G$ are equipped with a left Haar system
$\lambda=\set{\lambda^u : u\in G\z}$, and that if $G$ is \'etale, then
$\lambda^{u}$ is counting measure on $G^{u}$.  Let $C_c(G)$ be the
vector space of continuous and compactly supported complex-valued
functions on $G$, equipped with convolution and involution given for
$\gamma\in G$ and $f,g\in C_c(G)$ by
\[
  f*g(\gamma)=\int_G f(\beta)g(\beta\inv
  \gamma)\,\dd\lambda^{r(\gamma)}(\beta)\quad\text{and}\quad
  f^*(\gamma)=\overline{f(\gamma\inv )} .
\]
A $*$-homomorphism $L\colon C_c(G)\to B(H_L)$ into the bounded
operators on a Hilbert space $H_L$ is a \emph{representation} if it is
$I$-norm bounded.  Then the \emph{full $\Cst$-algebra} $\Cst(G)$ of
$G$ is the completion of $C_c(G)$ in the norm
\[
  \|f\|_{\Cst(G)}=\sup\set{\|L(f)\| : \text{$L$ is a representation of
      $C_c(G)$}}.
\]

We define another measure $\lambda_u$ on $G$ with support in $G_{u}$
by $\lambda_u(V)=\lambda^u(V\inv )$.  For each $u\in G\z $, define
$L^u\colon C_c(G)\to B(L^2(G_{u},\lambda_u))$ for $\gamma\in G_{u}$ by
\begin{equation*}
  \big(L^u(f)\xi\big)(\gamma)=\int_G f(\beta)\xi(\beta{\inv}
  \gamma)\, \dd\lambda^{r(\gamma)} (\beta).
\end{equation*}
Then $L^u$ is a representation of $C_c(G)$ on $L^2(G_{u},\lambda_u)$
and hence extends to a representation
$L^u\colon \Cst(G)\to B(L^2(G_{u},\lambda_u))$.  The \emph{reduced
  $\Cst$-algebra} $\Cst_\red(G)$ of $G$ is the completion of $C_c(G)$
in the norm
\[\|f\|_{\Cst_{r}(G)} =\sup\set{\|L^u(f)\| : u\in G\z }.\]

We write $\Ind_{G(u)}^G(\pi)$ for the representation of $\Cst(G)$
induced from a representation $\pi$ of the isotropy subgroup $G(u)$
\cite[Definition~5.12]{Williams:groupoid}.  If the orbits are
  locally closed in $\go$ (equivalently, when $\gugo$ is a $T_{0}$
  topological space), then every irreducible representation of
  $\cs(G)$ is of the form $\Ind_{G(u)}^{G}(\pi)$ for an irreducible
  representation of $G(u)$ for some $u\in\go$---see
  \cite[Theorem~5.35]{Williams:groupoid}.  If $G$ is amenable, then it
  is still the case that every primitive ideal of $\cs(G)$ is induced
  from a primitive ideal of $\cs(G(u))$.

\section{The quasi-orbit map}
\label{sec:quasi-orbit-map}

When $H$ is a locally compact group acting by automorphisms of a
$\Cst$-algebra $A$, the space of quasi-orbits $\mathcal Q$ and the
quasi-orbit map $k\colon \Prim A \to \mathcal{Q}$ play an important
role in the analysis of the ideal structure of $A\rtimes_\alpha H$,
see, for example, \cite{green:am78, Williams:Crossed}. In particular,
this quasi-orbit map is continuous and open.

Here we consider a groupoid $G$ and prove that an analogous
  quasi-orbit map---defined below---on $\Prim\Cst(G)$ is continuous,  and if $G$ is amenable 
  then it is open.  (When $G$ is \'etale, the continuity of the quasi-orbit map is proved in
\cite[Proposition~2.9]{Bonicke-Li-Nuclear-dim}.)  We start by
collecting the background needed to define the quasi-orbit map.

We write $\I(A)$ for the lattice of closed, two-sided ideals in $A$.
We equip $\I(A)$ with the topology with subbasic open sets
\begin{equation*}
  \O_{J}=\set{I\in\I(A):I\not\supset J}
\end{equation*}
where $J$ varies over all of $\I(A)$.  The relative topology on the
primitive ideal space $\Prim A $ of $A$ is the usual Jacobson
topology.
    
Let $F$ be a closed subset of a locally compact, Hausforff space
$X$. We write $I_F$ for the ideal in $C_{0}(X)$ of functions vanishing
on $F$.  The following is a small improvement of a special case of
\cite[Lemma~8.38]{Williams:Crossed} for $A=C_0(X)$.  That is, as in
Lemma~\ref{lem-seq-lift}, we can avoid passing to a subsequence.

\begin{lemma} \label{lem-fix-8.38} Let $X$ be a first countable,
  locally compact and Hausdorff space, and let $F_n, F$ be closed
  subsets of $X$. Suppose that the sequence $I_{F_{n}}\to I_F$ in
  $\I\bigl(C_{0}(X)\bigr)$ and let $x\in F$. Then there exist
  $x_{n}\in F_{n}$ such that $x_{n}\to x$ in $X$.
\end{lemma}

\begin{proof} Let $\sset{U_{k}}_{k=1}^{\infty} $ be a neighborhood
  basis at $x$ consisting of open sets such that
  $U_{k+1} \subset U_{k}$. Then for all $k$ we have
  $x\notin X\setminus U_k$ and hence $I_F\in \O_{I_{X\setminus U_k}}$.
  Since $I_{F_{n}}\in \O_{I_{X\setminus U_k}}$ eventually, there
  exists $m_{k}$ such that $n\ge m_{k}$ implies
  $I_{F_{n}}\in \O_{I_{X\setminus U_k}}$. We can take $m_{k+1}>m_{k}$
  for all $k$.  Notice that $F_{n}\cap U_{k}\not=\emptyset$ for
    $n\ge m_{k}$. 

  Now we choose $\{x_n\}$ as follows.  If $n<m_{1}$, let
  $x_{n}\in F_{n}$ be arbitrary.  If $m_{i}\le n < m_{i+1}$, choose
  $x_{n}\in F_{n}\cap U_{i}$.  To see that $x_n\to x$, let $U$ be any
  neighborhood of $x$.  There exists $i_{0}$ such that
  $U_{i_{0}}\subset U$.  Then if $n\ge m_{i_{0}}$, we have
  $x_{n}\in U_i\subset U$.  Thus $x_{n}\to x$.
\end{proof}

Recall that $G$ acts on $G\z$ by $\gamma\cdot s(\gamma)=r(\gamma)$. A
subset $F$ of $G\z$ is called \emph{$G$-invariant} if
$G\cdot F\subset F$; equivalently, $F$ is \emph{saturated} in the
sense that $F=r(s\inv(F))$.  We say $I\in \I(C_{0}(G\z))$ is
\emph{$G$-invariant} if $I=I_F$ where $F$ is a closed $G$-invariant
subset of $G\z$.  We let $\Ig(C_{0}(G\z))$ be the lattice of
$G$-invariant ideals in $\I(C_{0}(G\z))$.

\begin{prop}
  \label{prop-M-map} Let $G$ be a locally compact, Hausdorff
  groupoid, let $C_{b}(\go)$ be the bounded continuous functions
    on $\go$, and let $M(\cs(G))$ be the multiplier algebra.
  \begin{enumerate}
  \item There is an injective homomorphism
    $ V\colon C_b(G\z)\to M(\Cst(G)) $ such that for $h\in C_b(G\z)$
    and $f\in C_c(G)$ we have
    \begin{equation}\label{eq-M-map}(V(h)f)(\gamma)=h(r(\gamma))f(\gamma)
      \quad \text{and} \quad
      (fV(h))(\gamma)=h(s(\gamma))f(\gamma).
    \end{equation}
  \item If $J\in \I(\cs(G))$, then
    \[V^{*}(J)=\set{h\in C_0(G\z): \text{$V(h)a\in J$ for all
        $a\in \cs(G)$}}\] is an ideal in $\Ig(C_{0}(\go))$.
  \item The map $V^{*}\colon \I(\cs(G))\to \Ig(C_{0}(\go))$ is
    continuous and preserves intersections.
  \item If $L\colon \Cst(G)\to B(H)$ is a nondegenerate
    representation, then $V^{*}(\ker L)=\ker (\overline{L}\circ V)$,
    where $\overline{L}$ is the extension of $L$ to the multiplier
    algebra.
  \item If $P\in \Prim\cs(G)$, then there exists $u\in G\z$ such that
    $V^{*}(P)=I_{\overline{[u]}}$.
  \end{enumerate}
\end{prop}

\begin{proof} The existence of the homomorphism $V$ is proved in
  \cite[Lemma~1.48]{Williams:groupoid}. Here we make the additional
  observation that $V$ is injective.  To see this, suppose that
  $V(h)=0$. Let $u\in G\z$. There exist $\gamma\in G$ with
  $r(\gamma)=u$ and $f\in C_c(G)$ such that $f(\gamma)\neq 0$. Now
  $0=(V(h)f)(\gamma)=h(u)f(\gamma)$, that is, $h(u)=0$. Thus $h=0$,
  and $V$ is injective.

  By \cite[Proposition~9]{green:am78},
  $V^{*}\colon \I(\cs(G))\to \I(C_{0}(\go))$ is continuous and
  preserves intersections.  By
  \cite[Proposition~5.5]{Williams:groupoid},
  $\ker(\overline{L}\circ V)$ is $G$-invariant. So to see that $V^{*}$
  takes values in $\Ig(C_{0}(\go))$, as claimed, we show that
  $V^{*}(\ker L)=\ker (\overline{L}\circ V)$.  First let
  $h\in V^{*}(\ker L)$.  Then for all $a\in \cs(G)$, we have
  $V(h)a\in \ker L$.  Therefore $\overline{L}\circ V(h)L(a)=0$ for all
  $a\in \Cst(G)$.  Since $L$ is nondegenerate,
  $\overline{L}\circ V(h)=0$ and
  $V^{*}(\ker L) \subset \ker(\overline{L}\circ V) $.  Second, suppose
  that $h\in\ker(\overline{L}\circ V)$.  Then $V(h)a\in \ker L$ for
  all $a\in \cs(G)$, that is, $h\in V^{*}(\ker L)$, giving
  $\ker(\overline{L}\circ V)\subset V^{*}(\ker L)$.

  Finally, if $L$ is an irreducible representation, then
  $\ker(\overline{L}\circ V)=I_{\overline{[u]}}$ for some $u\in G\z$
  by \cite[Proposition~5.9]{Williams:Crossed}.
\end{proof}

We give the orbit space $\gugo$ the quotient topology.  In general,
$\gugo$ may not even be a $T_{0}$ topological space.  Then we work
with the \emph{$T_0$-isation},
$(G\backslash G\z)^{\sim}\coloneqq G\z/\!\!\sim$ defined in
\cite[Definition~6.9]{Williams:Crossed}, which is the quotient of
$\gugo$ where $G\cdot u\sim G\cdot v$ if and only if
$\overline{G\cdot u}=\overline{G\cdot v}$.  Then
$(G\backslash G\z)^{\sim}$ has the quotient topology, and the natural
map $k\colon G\z\to (G\backslash G\z)^{\sim}$ is open as well as
continuous by \cite[Lemma~4.9]{vanWyk-Williams}.

\begin{definition}\label{def-quasi-orbit-map}
  Let $G$ be a groupoid.  Define the \emph{quasi-orbit map}
  $p\colon \Prim \cs(G)\to %\gugo
  (G\backslash G\z)^{\sim}$ by $p(P)=k(u)$ where
  $V^{*}(P)=I_{\overline{[u]}}$.  (Notice that $p$ is well-defined by
  Proposition~\ref{prop-M-map}.)
\end{definition}

\begin{lemma}\label{qorbit-cts} Let~$G$ be a second-countable, locally compact, Hausdorff groupoid with a Haar system, and let $p\colon \Prim\Cst(G)\to (G\backslash G\z)^{\sim}$ be the quasi-orbit map.
Let  $P_{n}\to P$ in  $\Prim\cs(G)$ and let $u, u_n\in G\z$ be such that $p(P)=k(u)$ and $p(P_{n})=k(u_{n})$. Then there exist $\gamma_{n}\in G$ such that $\gamma_{n}\cdot u_{n}\to u_{0}$.  In particular, $p$ is continuous.
\end{lemma}

\begin{proof}
 Since $k$ is continuous, the last statement
  follows from the first.  By assumption,
  $V^{*}(P_{n})=I_{\overline{G\cdot u_{n}}}$. Since $V^{*}$ is
  continuous by \Cref{prop-M-map}, $V^{*}(P_{n})\to V^{*}(P)$ in
  $\I(C_{0}(\go))$.  By Lemma~\ref{lem-fix-8.38} there exist
  $v_{n} \in \overline {G\cdot u_{n}}$ such that $v_{n}\to u$.  But
  since $\go$ is metrisable, we can replace $v_{n}$ with
  $\gamma_{n}\cdot u_{n}$ for some $\gamma_{n}\in G$.
\end{proof}

\begin{thm} \label{qorbit-open}
  Let $G$ be a  second-countable, locally compact, Hausdorff groupoid  with a Haar system.    Then  the quasi-orbit map $p\colon \Prim\Cst(G)\to (G\backslash G\z)^{\sim}$ is  open.
\end{thm}

In an earlier version of this article, we proved that the quasi-orbit map is open when the  groupoid is amenable with a  Hausdorff orbit space. 
The proof of \Cref{qorbit-open}  below was shown to us by Sergey Neshveyev.  For the proof, we start with two lemmas that are not in the literature.
If $H$ is a group, we let $\lambda_{H}$ be the left-regular
representation of $H$ on $L^{2}(H)$.

\begin{lemma}\label{lem-wykwil3.2}  
Let $u_{n}\to u$ in $\go$.   Then
  $\ker\bigl( \Indx {u_{n}}(\lambdax {u_{n}})\bigr) \to
  \ker\bigl(\Indx u(\lambdax u)\bigr)$ in $\I(\cs(G))$.
\end{lemma}

  When $G$ has abelian isotropy, \Cref{lem-wykwil3.2} is a straightforward consequence of \cite[Lemmas~3.1 and 3.2]{vanWyk-Williams}.   Fortunately,
  many of the constructions in 
  \cite[\S3]{vanWyk-Williams} go through in the general case.  Let $\Sigma_0$ be  the space of closed subgroups of $G$ with the Fell topology.  Then we can
  still form the group bundle $\Sigma=\set{(H,t)\in\so\times G:t\in H}$, which has  a Haar system, and thus we can
  form the groupoid $\cs$-algebra $\cs(\Sigma)$.   If $\pi$ is any
  representation of $H\in\so$, then we get a representation $(H,\pi)$ of
  $\cs(\Sigma)$.   The constructions prior to
  \cite[Lemma~3.2]{vanWyk-Williams} proceed \emph{mutatis mutandis} so
  that $\mathsf X_{0}\coloneqq C_{c}(G*\so)$ completes to a right Hilbert
  $\cs(\Sigma)$-module $\mathsf X$ admitting a homomorphism of $\cs(G)$ into
  $\mathcal L(\mathsf X)$.  Therefore we can use the Rieffel induction
  process to
  induce representations $L$ of $\cs(\Sigma)$ to representations
  $\Ind_{\Sigma}^{G}(L)$ of $\cs(G)$. In particular,  
  $\ker L\mapsto \ker \bigl(\Ind_{\Sigma}^{G}(L)\bigr)$ is continuous
  from $\I(\cs(\Sigma))$ to $\I(\cs(G))$ just as in \cite[\S3]{vanWyk-Williams}.

\begin{proof}[Proof of Lemma~\ref{lem-wykwil3.2}]
For each $u\in\go$, the representation $(\{u\},1)$ is a character of
$\cs(\Sigma)$ such that $(\{u_{n}\},1)(a)\to (\{u\},1)(a)$ for all
$a\in \cs(\Sigma)$.  Thus $\ker  (\{u_{n}\},1) \to \ker
(\{u\},1)$ in $\I(\cs(\Sigma))$.    Hence
$\ker\bigl(\Ind_{\Sigma}^{G}(\{u_{n}\} ,1)\bigr) \to
\ker\bigl(\Ind_{\Sigma}^{G} (\{u\},1)\bigr)$ in $\I(\cs(G))$.

Since \cite[Lemma~3.1]{vanWyk-Williams} holds in our more general
setting of non-abelian isotropy subgroups,  $\Ind_{\Sigma}^{G}(\{u\},1)$ is equivalent to
$\Ind_{\{u\}}^{G}(1)$.   By induction in stages (see
\cite[Theorem~4]{Ionescu-Williams:irrep}) we have
\begin{equation*}
  \label{eq:9}
  \Ind_{\{u\}}^{G}1\sim \Indx u \bigl(\Ind_{\{u\}}^{G(u)}1\bigr)
  \sim \Indx u(\lambdax u),
\end{equation*}
and the lemma follows.
\end{proof}

\begin{lemma}
  \label{lem-irrep} Let $I\in \Prim\cs(G(u))$.   Then $\ker \bigl(
  \Indx u (I)\bigr) \in 
  \Prim \cs(G)$ and \[p\bigl(\ker\Indx u(I)\bigr)=k(u).\]
\end{lemma}

\begin{proof}
  Let $L$ be an irreducible  representation $L$ of
  $\cs(G(u))$ such that $I=\ker L$.   Then $\ker \bigl(
  \Indx u (L)\bigr) \in 
  \Prim \cs(G)$ by \cite[Theorem~5]{Ionescu-Williams:irrep}.
  Furthermore, \cite[Corollary~5.28]{Williams:groupoid} implies that
  the support $F$ of $\Indx u (L)$ is a subset of $\overline {G\cdot u}$.   But
   $G\cdot u\subset F$ by
  \cite[Lemma~5.25]{Williams:groupoid}, and hence $F=\overline{G\cdot u}=k(u)$.
\end{proof}

\begin{proof}[Proof of Theorem~\ref{qorbit-open}]
  Suppose that $p$ is not open.  Then there exist $J\in\Prim\Cst(G)$
  and an open neighbourhood $U$ of $J$ in $\Prim\Cst(G)$ such that $p(U)$
  is not an open neighbourhood of $p(J)$.  Since $G$ is amenable,
  \cite[Theorem~2.1]{Ionescu-Williams-EH} implies that
  every primitive ideal is induced from a primitive ideal of a stability
  subgroup.  Hence $J=\Indx u( I)$ for some $u\in \go$ and
  $I\in\Prim\Cst(G(u))$.

  Since
  $(\gugo)^{\sim}$ has the quotient topology and since $p(U)$ is not an open neighbourhood of $p(J)$, 
  $k\inv(p(U))$ is not an open neighbourhood of $u$  in $\go$. Hence there exist $\sset{u_n}$ in
  $\go\setminus k^{-1}(p(U))$ such that $u_{n}\to u\in k^{-1}(p(U))$.
  Note that $p\bigl(\ker(\Indx{u_{n}}(I_{n})\bigr) =k(u_{n})$ for any
  $I_{n}\in \Prim 
  \cs(G(u_{n}))$ by Lemma~\ref{lem-irrep}.
In particular, we can find  $I_n\in
  \Prim\Cst(G(u_n))$ such that 
  $\Indx{u_{n}}( I_n)\notin U$ for all $n$.

  Since $G$, and hence $G(u)$, is amenable, the left-regular
  representation $\lambdax u$ of $G(u)$ is
  faithful.  Therefore
  \begin{equation*}
    I\supset
    \{0\}=\ker \lambdax u.
  \end{equation*}
  Hence
  \begin{equation*}
    J=\Indx u(I)
    \supset \ker \Indx u(\lambdax u).
  \end{equation*}

  Since $u_{n}\to u$, Lemma~\ref{lem-wykwil3.2} implies that
  \[\ker\bigl(\Indx{u_{n}}(\lambdax{u_{n}}) \bigr)\to \ker
  \bigl(\Indx u(\lambdax u)\bigr)\] in $\mathcal
  I(\cs(G))$.  As in \cite[Lemma~4.7]{echeme:em11} for example,
  this implies 
  \begin{align*}
    \ker \bigl(\Indx u(\lambdax u)\bigr)
    &\supset \bigcap
      \ker\bigl(\Indx {u_{n}}(\lambdax {u_{n}})\bigr).
  \end{align*}

  On the other hand, using amenability again, for any $y\in\go$,
  \begin{equation*}
    \ker \lambdax y=\{0\} =\bigcap_{K\in \Prim \cs(G(y))} K,
  \end{equation*}
  and since Rieffel induction preserves direct
  sums \cite[Lemma~1.10]{Rieffel:Ind-C-star}
  %\aah{The Rieffel ref confirms this needs no proof :-)}
      it follows that
  \begin{equation*}
    \ker \bigl(\Indx y(\lambdax y)\bigr) = \bigcap_{K\in \Prim
      \cs(G(y))} \ker\bigl(\Indx y K\bigr).
  \end{equation*}
  Therefore
  \begin{align*}
    J\supset \ker\bigl(\Indx u(\lambdax u)\bigr)
    &\supset \bigcap_{n\in \mathbf N} \bigcap_{K\in\Prim(\cs(G(u_{n}))}
      \ker\bigl(\Indx{u_{n}}K \bigr),
  \end{align*}
 which  implies that $J$ is in the closure of
  \begin{align}
    \label{eq:7}
    \set{\ker\bigl( \Indx{u_{n}}(K)\bigr):\text{$n\in \mathbf N$ and $K\in \Prim
        \cs(G(u_{n}))$}}.
  \end{align}
  However,
  since $p\bigl( \Indx{u_{n}}(K)\bigr)=k(u_{n})$ by \Cref{lem-irrep},  the set at \eqref{eq:7} is contained in
  $\Prim\cs(G)\setminus U$, contradicting that $J\in U$.
\end{proof}

\section{Subhomogeneous \texorpdfstring{\cs}{C*}-algebras of
  groupoids}
\label{sec:subh-cs-algebr}

A $\Cst$-algebra $A$ is \emph{subhomogeneous} if there exists
$M\in\NN$ such that for every irreducible representation
$\pi\colon A\to B(H)$ on a Hilbert space $H$, the dimension $\dim(H)$
of $H$ is at most $M$.  In this section we characterise when a
$\Cst$-algebra of a groupoid is subhomogeneous, thus extending
\cite[Proposition~2.5]{Bonicke-Li-Nuclear-dim} from the \'etale
setting. Our greater generality shows, for example, that extensions of
certain \'etale groupoids by a trivial circle bundle that have arisen
in \cite{GWY:2017:DAD, CDaHGV, Bonicke-Li-Nuclear-dim} are
subhomogeneous.  When $\Cst(G)$ is subhomogeneous, we study its ideal
structure in \Cref{composition series-abelian} and find bounds on the
nuclear dimension of $\Cst(G)$ in \Cref{nuclear-dim-for-subhomog}.

 \begin{definition}\label{def-subhomo}
   A locally compact and Hausdorff groupoid is \emph{homogeneous}
   (\emph{subhomogeneous}) if its full groupoid $\Cst$-algebra is
   homogeneous (subhomogeneous).
 \end{definition}

 If $\Cst(G)$ is subhomogeneous, then so is its quotient $\Cst_r(G)$;
 we show below that subhomogeneous groupoids are amenable and so the
 two $\Cst$-algebras are isomorphic.  We start by noting what we know
 about subhomogeneous \emph{groups}.

 \begin{prop}\label{subhomog groups} Let $S$ be a locally compact
   group.
   \begin{enumerate}
   \item\label{subhomog groups1} Then $S$ is subhomogeneous if and
     only if $S$ has an open and abelian subgroup of finite index.
   \item\label{subhomog groups2} If $S$ has an open and abelian
     subgroup of finite index, then it has an open, abelian and normal
     subgroup of finite index.
   \item\label{subhomog groups3} If $S$ is subhomogeneous, then $S$ is
     amenable.
   \item\label{subhomog groups4} Suppose that $S$ is discrete.  Then
     $S$ is Type~I if and only if it is subhomogeneous.
   \end{enumerate}
 \end{prop}

 \begin{proof} Item~\cref{subhomog groups1} is
   \cite[Theorem~1]{Moore1972}.  For \cref{subhomog groups2}, let $H$
   be an open and abelian subgroup of $S$ with finite index.  Then
   $\set{sH s^{-1}:s\in S}$ is finite.  The intersection of groups of
   finite index has finite index.% \footnote{To see this, suppose that
   % $H_1, H_2$ are subgroups of a group $G$ with finite indices
   % $m, n$, respectively. Let $C(H_i)$ be a full set of cosets of
   % $H_i$. Then $f(C(H_1\cap H_2))\to C(H_1)\times C(H_2)$ where
   % $f(g(H_1\cap H_2))=(gH_1, gH_2)$ is well defined. Thus
   % $H_1\cap H_2$ has index at most $mn$.}
   Thus
   \begin{equation*}
     H_0\coloneqq\bigcap_{s\in S} s H s^{-1}
   \end{equation*}
   is a \emph{normal} open subgroup of finite index, giving
   \cref{subhomog groups2}.  Item~\cref{subhomog groups3} follows from
   \cref{subhomog groups1} and \cref{subhomog groups2} because the
   extension of an abelian group by a finite group is amenable.
   Item~\cref{subhomog groups4} follows from \cref{subhomog groups1},
   \cref{subhomog groups2} and \cite[Theorem~6]{Thoma1964} which says
   that a discrete group $S$ is Type~I if and only if it has an
   abelian and normal subgroup of finite index.
 \end{proof}

\begin{thm}\label{subhomg-characterisation}
  Let~$G$ be a locally compact, Hausdorff groupoid.
  \begin{enumerate}
  \item\label{subhomg-characterisation-1} Let $u\in G\z$ and let
    $L=\Ind_{G(u)}^{G}( \pi)$ be the representation of $\cs(G)$
    induced from the representation $\pi$ of the isotropy group
    $G(u)$.  Then $L$ is finite dimensional if and only if the orbit
    $[u]$ is finite and $\dim \pi <\infty$.  In that case,
    \begin{equation}
      \label{eq:1}
      \dim L=\big| [u] \big| \cdot \dim \pi.
    \end{equation}
  \item\label{subhomg-characterisation-2} Then $G$ is subhomogeneous
    if and only if there exists $M\in\NN$ such that for all $u\in G\z$
    we have $\big| [u]\big |\le M$ and $G(u)$ is $n$-subhomogenous
    with $n \le M$.
  \item\label{subhomg-characterisation-3} If $G$ is subhomogeneous,
    then $G$ is amenable.
  \end{enumerate}
\end{thm}

\begin{proof} 
For \cref{subhomg-characterisation-1}, let
  $u\in\go$ and let $\pi$ be a representation of the isotropy
  subgroup $G(u)$ on $H_{\pi}$. As described in
  \cite[Proposition~5.44]{Williams:groupoid}, the representation
  $\Ind_{G(u)}^{G} (\pi)$ of $\cs(G)$ induced from $\pi$ is equivalent
  to a representation $L^{u,\pi }$ acting on the Hilbert space
  $L^{2}_{\pi}(G_{u},\sigma_{u},H_{\pi})$ described in
  \cite[\S3.7]{Williams:groupoid}.   Further, it is observed in \cite[Lemma~3.47]{Williams:groupoid} that $L^{2}_{\pi}(G_{u},\sigma_{u},H_{\pi})$ is isomorphic
  to $L^{2}(G_{u}/G(u),\sigma_{u},H_{\pi})$. 
  
Here $\sigma_{u}$ is a
  Radon measure on $G_{u}/G(u)$ and
  \cite[Corollary~3.44]{Williams:groupoid} implies that $\sigma_{u}$
  has full support.  Since $\sigma_{u}$ has full support,
  $L^{2}(G_{u}/G(u),\sigma_{u})$ is finite dimensional if and only if
  $G_{u}/G(u)$ is finite, or, equivalently, if and only if the orbit
  $[u]$ of $u$ is finite.
Since $L^{2}(G_{u}/G(u),\sigma_{u},H_{\pi})$ is isomorphic to
  $L^{2}(G_{u}/G(u),\sigma_{u})\otimes H_{\pi}$ by, for example
  \cite[Lemma~I.13]{Williams:Crossed}, it follows that $L$ is
  finite-dimensional if and only if $|[u]|$ and $\dim H_{\pi}$ are finite, and if so the dimension of $ L$ is as in \eqref{eq:1}.

  For \cref{subhomg-characterisation-2}, first suppose that $G$ is
  $M$-subhomogeneous.  Then $\cs(G)$ is CCR.  By
  \cite[Theorem~3.5]{vanWyk-CCR}, the orbit space $\gmgo$ is
  $T_{1}$. Now every irreducible representation of $\cs(G)$ is induced
  from an isotropy group by
  \cite[Proposition~5.34]{Williams:groupoid}. Since every irreducible
  representation of $\Cst(G)$ has dimension at most $M$, it follows
  from \cref{subhomg-characterisation-1} that the orbits of $G$ are
  uniformly bounded by $M$ and that the isotropy subgroups are
  uniformly $M$-subhomogeneous.
    
  Conversely, suppose that the orbits are uniformly bounded and the
  isotropy subgroups are uniformly subhomogeneous.  Then the orbits
  are closed, that is, $\gmgo$ is $T_{1}$.  Then every irreducible
  representation of $\cs(G)$ is induced from a isotropy group.  Again,
  the result follows from \cref{subhomg-characterisation-1}.

  For \cref{subhomg-characterisation-3}, suppose that $G$ is
  subhomogeneous. Then by \cref{subhomg-characterisation-2}, $G$ has
  finite (hence closed) orbits and the isotropy groups are
  subhomogeneous.  Therefore $\gmgo$ is $T_{1}$, and all the isotropy
  groups are amenable by Proposition~\ref{subhomog groups}.  Therefore
  $G$ is amenable by \cite[Theorem~9.86]{Williams:groupoid}.
\end{proof}

The following example was told to us by Tyler Schultz.

\begin{example}\label{ex S3 acting} Consider the action of the symmetric group $S_3$ on
  $[-1,1]$ where $\omega\cdot x=x$ if $\omega$ is even and
  $\omega\cdot x=-x$ if $\omega$ is odd.  We let $G$ be the
  transformation-group groupoid.  Then the isotropy subgroup at $0$ is
  $S_3$, and if $x\neq 0$ then the isotropy subgroup is the
  alternating subgroup $A_3$ of $S_3$.  The isotropy subgroups do not
  vary continuously as $x\to 0$, and hence the spectrum of $\Cst(G)$
  is not Hausdorff by \cite[Proposition~1]{ech:tams1994}. The
  irreducible representations of $\Cst(S_3)$ are (the integrated forms
  of) the $1$-dimensional representations $1$ and $\sign$, and a
  $2$-dimensional representation $Q$. Since $A_3$ is abelian and
  isomorphic to $\ZZ/3\ZZ$, all its irreducible representations are
  $1$-dimensional. Since all the orbits are closed, all the
  irreducible representations of $\Cst(G)$ are induced: by
  Theorem~\ref{subhomg-characterisation} the $2$-dimensional ones are
  \[\set{\Ind_{G(u)}^G(\pi):u\neq 0,
      \pi\in\Cst(A_3)^\wedge}\cup\set{\Ind_{G(0)}^G(Q)}\] and the
  $1$-dimensional ones are
  \[
    \set{\Ind_{G(0)}^G(1), \Ind_{G(0)}^G(\sign)}.
  \]
  Thus $\Cst(G)$ is $2$-subhomogeneous.
\end{example}

Theorem~\ref{subhomg-characterisation} also applies to a large class
of examples arising as twists over subgroupoids of groupoids with
finite dynamic asymptotic dimension, as used in
\cite[Theorem~4.1]{CDaHGV}, \cite[Theorem~8.6]{GWY:2017:DAD}, and
\cite[Theorem~3.2]{Bonicke-Li-Nuclear-dim}.  We start by recalling the
definition of twist.

Let $G$ be a locally compact, Hausdorff groupoid, and view
$G\z \times\TT$ as a trivial group bundle with fibres $\TT$.  A
\emph{twist} $(\Sigma, \iota, \pi)$ over $G$ consists of a locally
compact and Hausdorff groupoid $\Sigma$ and groupoid homomorphisms
$\iota$, $\pi$ such that
\begin{equation*}
  \begin{tikzcd}
    \go\times\T \arrow[r,"\iota",hook] & \Sigma \arrow[r,"\pi",two
    heads] & G
  \end{tikzcd}
\end{equation*}
is a central groupoid extension; that is,
\begin{enumerate}
\item $\iota\colon G\z \times\TT\to \pi\inv (G\z)$ is a homeomorphism;
\item $\pi$ is continuous, open and surjective; and
\item $\iota(r(e), z)e=e\iota(s(e), z)$ for all $e\in \Sigma$ and
  $z\in\TT$.
\end{enumerate}

Suppose that $G$ is \'etale. Then $G\z=\Sigma\z$ is open in $G$ but it
is not open in $\Sigma$, and hence $\Sigma$ is never
\'etale. Nevertheless, $\Sigma$ has a natural Haar system: fix a
section $\mathfrak{c}\colon G\to \Sigma$ of $\pi$ and let
$\{\lambda^x\}$ be the Haar system on $\Sigma$ induced from the system
of counting measures on $G$, so that
\[
  \int f(e)\, \dd \lambda^x(e)=\sum_{\gamma\in xG}\int_\TT
  f\big(t\cdot\mathfrak{c}(\gamma)\big)\, \dd t
\]
for $f\in C_c(\Sigma)$---see \cite[Lemma~2.4]{CDaHGV}.

\begin{prop}\label{inverse image of groupoid with finite dad}
  Let $G$ be an \'etale groupoid and let $(\Sigma,\iota,\pi)$ be a
  twist over $G$. Suppose that there exists $M\in\NN$ such that
  $|G_u|\leq M$ for all $u\in G\z$. Then $\Sigma$ is a subhomogeneous
  groupoid with compact isotropy subgroups. Moreover, the twisted
  groupoid $\Cst$-algebra $\Cst(\Sigma;G)$ is subhomogeneous.
\end{prop}

\begin{proof} Let $u\in G\z=\Sigma\z$. The size of the orbit
  $[u]=r(G_{u})$ is bounded by $M$. The orbits of $G$ are the same as
  the orbits of $\Sigma$, and hence the orbits of $\Sigma$ are bounded
  by $M$.

  The size of the isotropy subgroup $G(u)$ of $G$ is also bounded by
  $M$, and hence each isotropy subgroup of $\Sigma$ is an extension by
  $\TT$ of a finite group with at most $M$ elements.  We have the
  exact sequence of groups
  \[
    \begin{tikzcd}
      0\arrow[r] & \iota(\sset u \times \T) \arrow[r] & \Sigma(u)
      \arrow[r] & G(u) \arrow[r]&0.
    \end{tikzcd}
  \]
  Since $G(u)$ is discrete, $\iota(\{u\}\times\TT)=\pi\inv(\{u\})$ is
  an open abelian subgroup of $\Sigma(u)$ with finite index at most
  $M$.  Since every irreducible representation of
  $\iota(\{u\}\times\TT)$ is $1$-dimensional, it follows from the
  \emph{proof} of \cite[Proposition~2.1]{Moore1972} that every
  irreducible representation of $\Sigma(u)$ is at most
  $M$-dimensional.  Moreover, $\Sigma(u)$ is compact because $\TT$ and
  $G(u)$ are \cite[Theorem~5.25]{Hewitt-Ross-I}.

  Since $u$ was arbitrary, the orbits of $\Sigma$ are uniformly
  bounded and the isotropy subgroups of $\Sigma$ are
  $M$-subhomogeneous. Thus $\Sigma$ is subhomogeneous by
  \Cref{subhomg-characterisation}.

  Finally, since $\Cst(\Sigma;G)$ is a quotient of the subhomogeneous
  $\Cst$-algebra $\Cst(\Sigma)$, it is also subhomogeneous.
\end{proof}

We will now work to understand the ideal structure of the
$C^*$-algebras of subhomogeneous groupoids.  The following is a
version of \cite[Lemma~2.10]{Bonicke-Li-Nuclear-dim} for non-\'etale
groupoids.

\begin{prop} \label{prop-BL-2.10} Let $G$ be a locally compact and
  Hausdorff groupoid. Let $q\colon \go\to\gugo$ be the orbit map, and
  define
  \begin{align*}
    \godm n&= \set{u\in \go:\card{{[u]}}\le n},\\
    \godp n&= \set{u\in \go:\card{{[u]}}\ge n},\\
    \god n&= \set{u\in \go:\card{{[u]}}= n}.
  \end{align*}
  \begin{enumerate}
  \item\label{prop-BL-2.10-1} For $n\ge0$, $\godm n$ is a closed and
    invariant subset of $\go$.
  \item\label{prop-BL-2.10-2} For $n\ge0$, $\godp n$ is an open and
    invariant subset of $\go$.
  \item\label{prop-BL-2.10-3} For $n\ge 1$, $\god n$ is invariant and
    locally closed, and hence is locally compact in $\go$.
  \item\label{prop-BL-2.10-4} For $n\ge 1$, $q(\god n)$ is locally
    closed in $\gugo$ and $q(\god n)$ is a Hausdorff subset of $\gugo$
    in the relative topology.
  \end{enumerate}
\end{prop}

To prove Proposition~\ref{prop-BL-2.10} we start with:

\begin{lemma}\label{helper X_n}
  Suppose that $u_{k}\to u_{0}$ in $\go$ with $u_{0}\in\godp n$.  Then
  $u_k\in\godp n$ eventually.
\end{lemma}

\begin{proof} By assumption, we can find distinct
  $\set{u_{0}^{1},\dots,u_{0}^{n}} \subset [u_{0}]$.  Since $q$ is
  open and $q(u_{k})\to q(u_{0}^{j})$ for $1\le j\le n$,
  Lemma~\ref{lem-seq-lift} implies that there are
  $\gamma_{k}^{j}\in G$ such that
  $\gamma_{k}^{j}\cdot u_{k} \to u_{0}^{j}$ for each $j$.  Since $\go$
  is Hausdorff, the elements
  $\set{\gamma_{k}^{1} \cdot u_{k},\dots,\gamma_{k}^{n}\cdot u_{k}}$
  are eventually distinct.  Hence we eventually have
  $\card{[u_{k}]}\ge n$.  The assertion follows.
\end{proof}

\begin{proof}[Proof of Proposition~\ref{prop-BL-2.10}] The sets in
  \cref{prop-BL-2.10-1}--\cref{prop-BL-2.10-3} are clearly invariant.

  If $n=0$, then $\godm n=\emptyset$ is closed. Fix $n\geq 1$. Let
  $u_k\to u$ in $\go$ with $u_k\in \godm n$. Suppose that
  $u\notin\godm n$. Then $u\in \godp {(n+1)}$. By Lemma~\ref{helper
    X_n}, $u_k\in \godp {(n+1)}$ eventually, a contradiction. Thus
  $u\in\godm n$, and $\godm n$ is closed. This gives
  \cref{prop-BL-2.10-1}.

  For \cref{prop-BL-2.10-2}, if $n=0$ then $\godp n=\go$ is open, and
  if $n\geq 1$ then $\go\setminus \godp n=\godm {(n-1)}$ is closed.

  For \cref{prop-BL-2.10-3}, let $n\geq 1$. Then
  $\god n=\godp n\cap \godm n$ is the intersection of an open and a
  closed set, hence is locally closed.  Locally closed subsets are
  locally compact by \cite[Lemma~1.26]{Williams:Crossed}.

  For \cref{prop-BL-2.10-4} we first show that $q(\god n)$ is locally
  closed in $\gugo$. Since $\godm n$ is closed and invariant, we
  have that
  $ q^{-1}\big( q(\go\setminus q(\godm n))\big)=\go\setminus \godm
    n $ is open.  Hence $q(\godm n)$ is closed. Since the orbit map
  is open, $q(\godp n)$ is open.  Since $\godm n$ and $\godp n$ are
  invariant we have that
  \[
    q(\god n)=q\big(\godm n\cap \godp n \big)=q(\godm n)\cap q(\godp
    n )
  \]
  is the intersection of open and closed sets.  Therefore it is
  locally closed.

  To see that $q(\god n)$ is Hausdorff, suppose that $\{q(u_{k})\}$
  converges to both $q(u)$ and $q(v)$ in $q(\god n)$.  We need to see
  that $q(u)=q(v)$.  If not, then $[u]$ and $[v]$ are disjoint, finite
  sets in $\go$. Then there are disjoint open sets $U$ and $V$ in
  $\go$ such that $[u]\subset U$ and $[v]\subset V$.  Let
  $[u]=\set{u^{1},\dots,u^{n}}$.  Since $q$ is open,
  Lemma~\ref{lem-seq-lift} implies that there are
  $\gamma_{k}^{j}\in G$ such that $\gamma_{k}^{j}\cdot u_{k}\to u^{j}$
  for $1\le j\le n$.  Since $\go$ is Hausdorff, we can assume that the
  $\gamma_{k}^{j} \cdot u_{k}$ are eventually distinct and contained
  in $U$.  Since $\card{[u_{k}]}=n$ for all $k$, we eventually
  have $[u_{k}]$ in 
  $U$.  But again by Lemma~\ref{lem-seq-lift}, there are
  $\eta_{k}\in G$ such that $\eta_{k}\cdot u_{k}\to v$.  But then
  $[u_{k}]$ eventually meets $V$ which contradicts
  $U\cap V=\emptyset$.
\end{proof}

The following proposition is broadly stated in
\cite[Proposition~3.6.3]{Dixmier}; our statement is more explicit
about the homeomorphisms.  In \Cref{composition series-abelian} we
apply it to the $\Cst$-algebra of a second countable, locally compact
and Hausdorff groupoid with a Haar system and abelian isotropy
subgroups.

Let $A$ be a $\Cst$-algebra. We write $\hat A$ for the spectrum of
$A$, and for $k\in\NN$ we write $\hat A_{k}$ for the irreducible
representations $\pi\colon A\to B(H_\pi)$ with $\dim(H_\pi)=k$, and
$\Prim_k A $ for the set of $P\in\Prim A $ such that $P=\ker\pi$ for
$\pi\in \hat A_{k}$.

\begin{prop}\label{composition series}
  Let $j\in\NN$ and $M_0, M_1,\dots, M_j\in\NN$ such that
  $0= M_0<\dots<M_j$. Let $A$ be a subhomogeneous $\Cst$-algebra with
  $\hat A_M=\emptyset$ unless $M=M_n$ for some
  $n\in\{0,\dots,j\}$. For $0\leq n\leq j$, let $I_n$ be the ideal
  \[
    I_n=\bigcap\{\ker\pi\in\Prim A: \dim(\pi)\leq M_n\}
  \]
  of $A$.  Then
  \[
    \{0\}=I_j\subset\cdots\subset I_n\subset I_{n-1}\cdots \subset
    I_1\subset I_0=A
  \]
  is a composition series of ideals of $A$. For $1\leq n\leq j$, let
  $q_n\colon I_{n-1}\to I_{n-1}/I_n$ be the quotient map. Then
  \begin{enumerate}
  \item\label{composition series-1}
    $\ker\pi\mapsto \ker(\pi|_{I_{n-1}})$ is a homemomorphism of
    $\{\ker\pi\in\Prim A \colon\dim(\pi)\geq M_n\}$ onto
    $\Prim I_{n-1} $.
  \item\label{composition series-2} $I_{n-1}/ I_n$ is
    $M_n$-homogeneous, and the map
    \[Q\mapsto \overline{ Aq_n ^{-1}(Q) A}\] is a
    homeomorphism of $\Prim I_{n-1}/I_n $ onto $\Prim_{M_n}  A$.
  \end{enumerate}
  In particular, $\Prim_M A$ is locally compact and Hausdorff for
  every $M\in\NN$.
\end{prop}

\begin{proof} Since homogeneous \cs-algebras are GCR, we identify
  their primitive ideal spaces and spectra.  By
  \cite[Proposition~3.6.3(i)]{Dixmier}, the set
  \begin{equation}\label{set F}
    F_n\coloneqq\{ \pi\in \hat A: \dim(\pi)\leq M_n\}
  \end{equation}
  is closed.  That $\Prim I_{n-1}/I_n $ is homeomorphic to
  $\Prim_{M_n} A $ is stated in \cite[Proposition~3.6.3(ii)]{Dixmier},
  but the formula for the homemomorphism is not part of the statement
  nor the proof.

  Let $\rho$ be an irreducible representation of $I_{n-1}/I_n$ and let
  $q_n\colon I_{n-1}\to I_{n-1}/ I_n$ be the quotient map.  Then
  $\rho=\tilde\phi$ where $\phi$ is an irreducible representation of
  $I_{n-1}$ such that $\phi=\tilde\phi\circ q_n$.  Indeed, by, for
  example, \cite[Proposition~A.27]{tfb},
  \[
    \Psi_1^{n-1}\colon \left(I_{n-1}/I_n\right)^\wedge \to \{\phi\in
    (I_{n-1})^{\wedge}\colon \phi|_{I_n}=0\}, \quad \tilde\phi\mapsto
    \phi
  \]
  is a homeomorphism.  We have $\dim(\tilde\phi)=\dim(\phi)$, and we
  claim that $\dim(\phi)=M_n$.

  By \cite[Proposition~A27]{tfb} again,
  \[
    \Psi_2^{n-1}\colon \set{\pi\in \hat A\colon \pi|_{I_{n-1}}\neq
      0}\to (I_{n-1})^{\wedge},\quad \pi\mapsto \pi|_{I_{n-1}}
  \]
  is a homeomorphism.  Here we have that
  \begin{align*}
    \set{\pi\in \hat A\colon \pi|_{I_{n-1}}\neq 0}
    &=\set{\pi\in \hat A\colon\dim(\pi)> M_{n-1}} 
      =\set{\pi\in \hat A\colon\dim(\pi)\geq M_{n}}.
  \end{align*}
  This establishes \cref{composition series-1}.

  Further, $\phi=\pi|_{I_{n-1}}$ where $\pi$ has dimension at least
  $M_n$, and since $I_{n-1}$ is an ideal we have $\dim(\phi)\geq M_n$
  as well.  But $I_n\subset\ker(\phi)$ gives $\dim(\phi)\leq M_n$
  because the set $F_n$ in Equation~\ref{set F} is closed. Thus
  $\dim(\phi)=M_n$ as claimed in the first part of \cref{composition
    series-2}.

  The inverse of $\Psi_2^{n-1}$ sends
  $\phi\colon I_{n-1}\to B(H_\phi)$ to the unique extension
  $\bar\phi\colon A\to B(H_\phi)$ such that
  $\phi(ai)=\bar\phi(a)\phi(i)$ for $a\in A$ and $i\in I_{n-1}$. Thus
  $\dim(\bar\phi)=\dim(\phi)=M_n$.  Now
  $(\Psi_2^{n-1})^{-1}\circ\Psi^{n-1}_1$ maps
  $\tilde\phi\mapsto\bar\phi$ and has range contained in $\Prim_{M_n} A $.
  To see $(\Psi_2^{n-1})^{-1}\circ\Psi^{n-1}_1$ is onto
  $\Prim_{M_n} A $, let $\pi\in \Prim_{M_n} A $, then $\pi|_{I_{n-1}}$
  has dimension $M_n$. Thus $\pi|_{I_{n}}=0$ and hence
  $\pi|_{I_{n-1}}$ factors through the quotient. Thus
  $(\Psi_2^{n-1})^{-1}\circ
  \Psi^{n-1}_1\big(\widetilde{\pi|_{I_{n-1}}}\big)=\pi$. This gives
  \cref{composition series-2}.

  Homogeneous $\Cst$-algebras have Hausdorff spectrum, giving the last
  statement.
\end{proof}

Now we apply \Cref{composition series} to $A=\Cst(G)$.  If
$M_n\in\NN$, then $G\z_{M_n}$ is locally closed by
\Cref{prop-BL-2.10}, and hence $G|_{G\z_{M_n}}$ is locally
compact. Thus writing $\Cst(G|_{G\z_{M_n}})$ in \Cref{composition
  series-abelian} below makes sense.

\begin{cor}\label{composition series-abelian} Let $j\in\NN$ and $M_0,
  M_1,\dots, M_j\in\NN$ such that $0= M_0<\dots<M_j$.  Let $G$ be a
  locally compact, Hausdorff and subhomogeneous groupoid with abelian
  isotropy subgroups.  Suppose that $\Prim_M\Cst(G)=\emptyset$ unless
  $M=M_n$ for some $n\in\{0,\dots,j\}$. Let $n\in \{0, \dots, j\}$ and
  let $I_n$ be the ideal
  \[
    I_n=\bigcap\set{\ker\pi\in\Prim\Cst(G)\colon \dim(\pi)\leq M_n}
  \]
  of $\Cst(G)$. Then $I_{n-1}$ is isomorphic to
  $\Cst(G|_{\godp{M_n}})$ and the quotient $I_{n-1}/I_n$ is isomorphic
  to $\Cst(G|_{\god {M_n}})$ for $1\leq n\leq j$.  Further,
  $\Prim\big( \Cst(G|_{\god{M_n}}\big)$ is homemomorphic to
  $\Prim_{M_n} \Cst(G)$.
\end{cor}

\begin{proof}
  Since $G$ is subhomogeneous, and hence CCR,
  every irreducible representation of
  $\Cst(G)$ is of the form $\Ind_{G(u)}^G(\sigma)$ for some $u\in\go$
  and $\sigma\in G(u)^{\wedge}$.  Since the isotropy is abelian,
  $\dim(\sigma)=1$ and
  \Cref{subhomg-characterisation}\cref{subhomg-characterisation-1}
  implies that $\dim\big(\Ind_{G(u)}^G(\sigma)\big)\leq M_{n-1}$ if
  and only if $u\in \godm{M_{(n-1)}}$.

  By \Cref{prop-BL-2.10}, $\godp{M_n}$ is open and invariant. Thus
  $H_n\coloneqq G|_{\godp{M_n}}$ is an open subgroupoid of $G$, and
  inclusion and extension by $0$ extends to an isometric
  homomorphism \[\iota\colon \Cst(H_n)\to \Cst(G).\] Since
  $\godp{M_n}$ is an invariant subset of the unit space, the range of
  $\iota$ is an ideal of $\Cst(G)$.
 
  To see that $\iota$ has range $I_{n-1}$, fix
  $u\in \godm{M_{(n-1)}}$, $\sigma\in G(u)^{\wedge}$, and
  $f\in C_c(H_n)$. For $\xi\otimes h\in C_c(G_{u})\otimes H_\sigma$ we
  have
  \[\Ind_{G(u)}^G(\sigma)(\iota(f))(\xi\otimes h)=\iota(f)*\xi\otimes
    h.\] Here
  \[
    (\iota(f)*\xi)(\gamma)=\int_G
    \iota(f)(\alpha)\xi(\alpha^{-1}\gamma)\,
    \dd\lambda^{r(\gamma)}(\alpha)=0
  \]
  because in the integrand $s(\gamma)=u\in \godm{M_{(n-1)}}$ implies
  $\alpha\in \godm{M_{(n-1)}}$, and then $\iota(f)(\alpha)=0$.  Thus,
  by definition of $I_{n-1}$, $\range\iota$ is an ideal contained in
  $I_{n-1}$.
    
  To see that $\iota$ is surjective, notice that
  since the orbits in $H_n$ are closed, every primitive ideal of
  $\Cst(H_n)$ is induced. Moreover, for $u\in \godp{M_n}$, the orbits
  and isotropy subgroups in $H_n$ and $G$ coincide.  Thus
  $\Ind_{H_n(u)}^{H_n}(\sigma)\mapsto \Ind_{G(u)}^{G}(\sigma)$ is a
  bijection of $\Cst(H_n)^\wedge$ onto
  $\{\pi\in\Cst(G)^\wedge\colon \dim(\pi)\geq M_n\}$ which is
  homeomorphic to $(I_{n-1})^\wedge$ by \Cref{composition series}. It
  follows that $\Cst(H_n)$ is isomorphic to $I_{n-1}$.

  Now $I_n$ is isomorphic to $\Cst(H_{n+1})$; since
  $\godp{M_{(n+1)}}$ is an open invariant subset of $\godp{M_{n}}$,
  the statement about the quotient is immediate from, for example,
  \cite[Proposition~5.1]{Williams:groupoid}.  The assertion about
  primitive ideal spaces follows from part~(b) of
  Proposition~\ref{composition series}.
\end{proof}

\begin{remark} There is no staightforward analogue of  \Cref{subhomg-characterisation} when the isotropy subgroups are not abelian. To see this,  let $G=[-1,1]\times S_3$ be the transformation-group groupoid of \Cref{ex S3 acting}. The isotropy subgroup at $0$ is all of $S_3$.  The ideal
\[
I_1=\cap\{\ker\pi\in\Prim\Cst(G) :\dim\pi\leq 1\}
\]
has spectrum
\[
\widehat{I_1}=\{ \Ind_{G(u)}^G(\sigma):u\neq 0, \sigma\in \widehat{A_3} \}\cup\{ \Ind_{G(0)}^G(Q)\}.
\]
If $I_1$ were of the form $\Cst(G|_U)$ for an open, invariant subset $U$ of $\go$, then $U$ would have to be $[-1,1]$, which would imply $I_1=\Cst(G)$. 
\end{remark}

When $G$ is subhomogeneous, all orbits are closed. Then the
quasi-orbit map from \Cref{def-quasi-orbit-map} takes values in
$\gugo$, and $p(\Ind_{G(u)}^G(\sigma))=G\cdot u$.  Hence \Cref{prim M
  n}\cref{prim M n item 2} is a non-\'etale version of
\cite[Proposition~2.11]{Bonicke-Li-Nuclear-dim}.

\begin{prop}\label{prim M n}
  Let $G$ be a locally compact, Hausdorff groupoid that is
  subhomogeneous, and let $n$ and $M$ be positive integers such that
  $n\mid M$. Let \[p_M\colon \Prim_M\Cst(G)\to \gugo\] be the
  restriction of the quasi-orbit map.
  \begin{enumerate}
  \item\label{prim M n item 1} Then $p_M^{-1}(G\backslash G\z_{n})$ is
    locally compact and Hausdorff.
  \item\label{prim M n item 2} Fix $u\in G\z_{n}$. Then
    $\Ind_{G(u)}^G(\pi) \mapsto \pi$ is a homemorphism of
    $p_M^{-1}(G\cdot u)$ onto $\Prim_{M/n}\Cst(G(u))$.
  \end{enumerate}
\end{prop}

\begin{proof} By \Cref{prop-BL-2.10}, $G\backslash G\z_{n}$ is locally
  closed in $\gugo$. Since $p_M$ is continuous by \Cref{qorbit-cts}, $p_M^{-1}(G\backslash G\z_{n})$ is a locally
  closed subset of a locally compact space, hence is locally
  compact. Since $\Prim_M\Cst(G)$ is Hausdorff by \Cref{composition
    series}, so is the subset $p_M^{-1}(G\z_{n}/G)$. This gives
  \cref{prim M n item 1}.

  Since $G$ is subhomogeneous, the orbit $[u]$ is closed in $G\z$. Let
  $q\colon \Cst(G)\to \Cst(G|_{[u]})$ be the quotient map.  Since
  $u\in G\z_n$, by \Cref{subhomg-characterisation} every element of
  $p_M^{-1}(G\cdot u)$ is of the form $\Ind_{G(u)}^G(\pi)$ where
  $\pi\in \Cst(G(u))^\wedge$ has dimension $M/n$. Further,
  $\Ind_{G(u)}^G(\pi)$ factors through the representation
  $\Ind_{G(u)}^{G|_{[u]}}(\pi)$ of the quotient $\Cst(G|_{[u]})$ of
  $\Cst(G)$, that is,
  $\Ind_{G(u)}^G(\pi)=\Ind_{G(u)}^{G|_{[u]}}(\pi)\circ q$---see 
  \cite[Corollary~5.29]{Williams:groupoid}. In particular,
  \[p_M^{-1}(\{G\cdot u\})=\{P\in\Prim\Cst(G)\colon \ker q\subset
    P\},\] and is homeomorphic to $\Prim\Cst(G|_{[u]})$.  By the
  Equivalence Theorem
  \cite[Theorem~3.1]{Muhly-Renault-Williams-Equivalence}, $\Cst(G(u))$
  and $\Cst(G|_{[u]})$ are Morita equivalent.  Let $Y$ be the
  associated $\Cst(G(u))$--$\Cst(G|_{[u]})$ imprimitivity bimodule and
  \[\ibind Y\colon\Prim\Cst(G|_{[u]})\to \Prim\Cst(G(u))\] the
  Rieffel homeomorphism.  Then the composition gives
  $\ibind Y(\Ind_{G(u)}^{G|_{[u]}}(\pi))=\pi$, as needed.  This gives
  \cref{prim M n item 2}.
\end{proof}

We are now ready to discuss the nuclear dimension of $\Cst$-algebras
of subhomogeneous groupoids. We start by recalling the definitions of
topological and nuclear dimension.

Let $X$ be a topological space.  An open cover of $X$ has order $m$ if
$m$ is the least integer such that each element of $X$ belongs to at
most $m$ elements of the cover. The \emph{topological covering
  dimension} of $X$, written $\dim(X)$, is the smallest integer~$N$
such that every open cover of $X$ admits an open refinement of order
$N+1$ (see, for example, \cite[\S1.1]{Coornaert} or
  \cite[\S1.6]{Engelking}).

Let $A$ be a $\Cst$-algebra and let $d\in \NN$. Then $A$ has
\emph{nuclear dimension} at most $d$, as defined in
\cite[Definition~2.1]{WinterZacharias:nuclear} and written
$\dim_\nuc(A)\leq d$, if for any finite subset $\Ff\subset A$ and
$\epsilon>0$, there exist a finite-dimensional $\Cst$-algebra
$F=\bigoplus_{i=0}^d F_i$, a completely positive, contractive map
$\psi\colon A\to F$ and a completely positive map $\phi\colon F\to A$
such that $\phi|_{F_i}$ is contractive and order zero for
$0\leq i\leq d$, and for all $a\in \Ff$,
\[\|\phi(\psi(a))-a\|<\epsilon.\]

The following is a non-\'etale version of
\cite[Theorem~2.12]{Bonicke-Li-Nuclear-dim}.  To make the formulas
look nicer, we write $\dim\plusone_\nuc(A)$ for $\dim_\nuc(A)+1$ and
$\dim\plusone(X)$ for $\dim(X)+1$.

\begin{thm}\label{nuclear-dim-for-subhomog}
  Let $G$ be a locally compact, Hausdorff and subhomogeneous
  group\-oid. Then
  \[
    \dim_\nuc\plusone(\Cst(G))\leq \dim\plusone(G\z)\sup_{u\in
      G\z}\dim\plusone(\Prim\Cst(G(u))).
  \]
\end{thm}

\begin{proof} Since $\Cst(G)$ is separable and subhomogeneous, by
  \cite[\S1.6]{Winter:decomposition-rank} we have
  \[
    \dim_\nuc(\Cst(G))= \max_M\{\dim(\Prim_M\Cst(G))\}.
  \]
  Let $M\in\NN$ such that $\Prim_M\Cst(G)\neq \emptyset$.  Write
  $p_{M}$ for $p$ restricted to $\Prim_M\Cst(G)$. Then
  \[
    \Prim_{M}\Cst(G)=\bigsqcup_{n\mid M}p\inv_M(G\backslash G\z_n),
  \]
  and hence
  \[
    \dim \big(\Prim_{M}\Cst(G)\big)=\max_{n\mid
      M}\dim(p_{M}\inv(G\backslash G\z_n)).
  \]
  By \Cref{prop-BL-2.10} and \Cref{prim M n}, respectively, each
  $G\backslash G\z_n$ and $p_{M}\inv(G\backslash G\z_n)$ is locally
  compact and Hausdorff.  In particular,
  $p_{M}\colon p_{M}^{-1}( G\backslash G\z_n)\to G\backslash G\z_n$ is
  a continuous map between locally compact, Hausdorff spaces. Thus we
  can view $C_0(p_M^{-1}(G\backslash G\z_n))$ as a
  $C_0(G\backslash G\z_n)$-algebra as in
  \cite[Proposition~5.37]{Williams:groupoid}.  Now apply
  \cite[Lemma~3.3]{HW:2017:NDim-homeo} to get
  \begin{align*}
    \dim\plusone(p_{M}^{-1}(G\backslash G\z_n))
    &=\dim_\nuc\plusone(C_0(p_{M}^{-1}(G\backslash G\z_n)))\\
    &\leq \dim\plusone(G\backslash G\z_n)\sup_{G\cdot u\in
      G\backslash G\z_n}\dim_\nuc\plusone (C_0(p_{M}^{-1}(G\cdot u)))\\ 
    &=\dim\plusone(G\backslash G\z_n)\sup_{u\in G\z_n}\dim\plusone
      (p_{M}^{-1}(G\cdot x))\\ 
    &=\dim\plusone(G\backslash G\z_n) \sup_{u\in G\z_n}\dim\plusone
      \big(\Prim_{M/n}\Cst(G(u))\big) 
  \end{align*}
  using \Cref{prim M n} at the last step.  Since $\Cst(G(u))$ is
  separable and subhomogeneous, by
  \cite[\S1.6]{Winter:decomposition-rank} we have
  \[\dim \big(\Prim_{M/n}\Cst(G(u))\big)\leq \dim
    \big(\Prim\Cst(G(u))\big).\] The restricted orbit map
  $q\colon G\z_n\to G\backslash G\z_n$ is an open, surjective map
  between spaces that are second countable, locally compact and
  Hausdorff (hence metrisable); since $q^{-1}(q(x))$ is finite for all
  $x\in G\z_n$, it follows from \cite[Theorems~1.12.7
  and~1.7.7]{Engelking} that $\dim(G\z_n)=\dim(G\backslash G\z_n)$.
  Finally, since $G\z_n$ is a subset of the metrisable space $G\z$, we
  get that \[\dim(G\backslash G\z_n)=\dim(G\z_n)\leq\dim(G\z).\] Thus,
  for each $M$,
  \[
    \dim\plusone(p_{M}^{-1}(G\backslash G\z_n))\leq
    \dim\plusone(G\z)\sup_{u\in G\z}\dim\plusone
    \big(\Prim\Cst(G(u))\big),
  \]
  and the theorem follows.
\end{proof}

We can say more when the isotropy subgroups are abelian or compact
(\Cref{nuclear-dim-for-subhomog-abelian-or-compact-isotropy}) and for
non-\'etale extension of certain \'etale subhomogeneous groupoids
(\Cref{cor-nuc-dim-twist}).

\goodbreak
\begin{cor}\label{nuclear-dim-for-subhomog-abelian-or-compact-isotropy}
  Let $G$ be a locally compact, Hausdorff, subhomogeneous group\-oid.
  \begin{enumerate}
  \item Suppose that the isotropy subgroups are isomorphic and
    homeomorphic to a subgroup of an abelian group $S$.  Write
    $\hat S$ for the dual group of $S$.  Then
    \[
      \dim_\nuc\plusone(\Cst(G))\leq
      \dim\plusone(G\z)\dim\plusone(\hat S).
    \]
  \item\label{nuclear-dim-for-subhomog-abelian-or-compact-isotropy2}
    Suppose that the isotropy subgroups are compact. Then
    \[
     \dim_\nuc(\Cst(G))\leq \dim (G\z).
    \]
  \end{enumerate}
\end{cor}
\begin{proof} First assume that the isotropy subgroups are isomorphic
  to a subgroup of an abelian group $S$.  We identify each $G(u)$ with
  its homeomorphic copy in $S$ for all $x\in G\z$.  Since $S$ is
  abelian, each $\Cst(G(u))\cong C_0(G(u)^{\wedge})$. Then
  $\Prim\Cst(G(u))$ is homeomorphic to $G(u)^{\wedge}$. Since
  $G(u)^{\wedge}$ is homeomorphic to the quotient
  $\widehat{S}/G(u)^\perp$ of $\widehat{S}$ by the annihilator
  subgroup $G(u)^\perp$ we obtain
  \[
    \dim(\widehat{S})=\dim \big(\widehat{S}/G(u)^\perp
    \big)+\dim(G(u)^\perp )\geq\dim \big(\widehat{S}/G(u)^\perp
    \big)=\dim(G(u)^{\wedge})
  \]
  by \cite[Theorem~2.1]{Nagami}.  Now by
  \Cref{nuclear-dim-for-subhomog} we have
  \begin{align*}
    \dim_\nuc\plusone(\Cst(G))
    &\leq \dim\plusone(G\z)\sup_{u\in G\z}\dim\plusone(G(u)^{\wedge})\\
    &\le\dim\plusone(G\z)\dim\plusone(\widehat{S}).
  \end{align*}

  Second, suppose that all the isotropy subgroups are compact. Then for
  all $u\in G\z$ we have $\dim(\Prim\Cst(G(u)))=0$ because
  $\Prim \Cst(G(u))$ is discrete by \cite[18.4.3]{Dixmier}. Thus
  $\dim_\nuc\plusone(\Cst(G))\leq \dim\plusone(G\z) $ by
  \Cref{nuclear-dim-for-subhomog}.
\end{proof}

\begin{example}  Let $G=[-1,1]\times S_3$ be the transformation-group groupoid of \Cref{ex S3 acting}. By \Cref{nuclear-dim-for-subhomog-abelian-or-compact-isotropy}, $\dim_\nuc(\Cst(G))\leq\dim([-1,1])=1$. Now $\dim_\nuc(\Cst(G))=1$.
\end{example}

\begin{cor}\label{cor-nuc-dim-twist}
  Let $G$ be an \'etale groupoid and let $(\Sigma,\iota,\pi)$ be a
  twist over $G$. Suppose that there exists $M\in\NN$ such that
  $|G_{x}|\leq M$ for all $x\in G\z$.  Then
  \[
    \dim_\nuc\plusone(\Cst(\Sigma))\leq \dim\plusone(G\z).
  \]
  Furthermore,
  $\dim_\nuc\plusone(\Cst(\Sigma; G))\leq \dim\plusone(G\z)$.
\end{cor}

\begin{proof}
  By \Cref{inverse image of groupoid with finite dad}, $\Sigma$ is a
  subhomogeneous groupoid with compact isotropy subgroups. The
  statement about $\Cst(\Sigma)$ now follows from
  \Cref{nuclear-dim-for-subhomog-abelian-or-compact-isotropy}
  \cref{nuclear-dim-for-subhomog-abelian-or-compact-isotropy2} and the
  statement about the quotient $\Cst(\Sigma;G)$ of $\Cst(\Sigma)$
  follows from \cite[Proposition~2.9]{WinterZacharias:nuclear}.
\end{proof}

\section{Nuclear dimension of \texorpdfstring{\cs}{C*}-algebras of
  \'etale groupoids with large isotropy subgroups}
\label{sec:nuclear-dimension-cs}

Here we consider the nuclear dimension of $\Cst$-algebras of groupoids
that are not subhomogeneous but still have an abundance of
subhomogeneous $\Cst$-subalgebras.  In particular, we consider
groupoids that are extensions of groupoids with finite dynamic
asymptotic dimension.  We start by recalling the definition of dynamic
asymptotic dimension from \cite{GWY:2017:DAD}.

\begin{definition}[{\cite[Definition
    5.1]{GWY:2017:DAD}}]\label{defn-dad} Let $G$ be an \'etale groupoid. 
  Then $G$ has \emph{dynamic asymptotic dimension} $d\in\NN$ if $d$ is
  the smallest natural number with the property that for every open
  and precompact subset $W\subset G$, there are open subsets
  $U_0, U_{1},\dots, U_d$ of $G\z $ that cover $s(W)\cup r(W)$ such
  that for $i\in\set{0,\dots, d}$ the set
  \[\set{\gamma \in W : s(\gamma), r(\gamma)\in U_i}\] is contained in
  a precompact subgroupoid of~$G$. If no such $d$ exists, then we say
  that $G$ has infinite dynamic asymptotic dimension. We write
  $\dad(G)$ for the dynamic asymptotic dimension of $G$ and use
  $\dad\plusone(G)$ for $\dad(G)+1$.
\end{definition}

Our main tool is the following proposition from \cite{CDaHGV}.

\begin{prop} [{\cite[Proposition~4.2]{CDaHGV}}]
  \label{nuclear dim via sublagebras}
  Let $A$ be a unital $\Cst$-algebra and let $X\subset A$ be such that
  $\operatorname{span}(X)$ is dense in $A$. Let $d,n\in \NN$. Suppose
  that for every finite $\Ff\subset X$ and every $\epsilon>0$ there
  exist $\Cst$-subalgebras $B_0,\ldots ,B_d$ of $A$ and
  $b_0,\ldots ,b_d\in A$ of norm at most $1$ such that
  $\dim_{\nuc}(B_i)\leq n$ and $b_{i}\mathcal{F}b_{i}^*\subset B_{i}$
  for $0\leq i\leq d$, and
  $\big\|x-\sum_{i=0}^d b_ixb_i^*\big\|<\epsilon$ for all $x\in\Ff$.
  Then the nuclear dimension of $A$ is at most $(d+1)(n+1)-1$.
\end{prop}

If we assume that the isotropy subgroups of $G$ vary continuously,
then the quotient $\gmiso$ is a locally compact, Hausdorff and \'etale
groupoid by \Cref{lem hypotheses for lc quotient}.  Furthermore, the
quotient map $\rho\colon G\to \gmiso$ is open. In particular,
$\dad(\gmiso)$ is defined in \Cref{main thm} below.

\begin{remark}\label{rem V_r}
  Let $V\colon C_b(G\z)\to M(\Cst(G))$ be the map from
  \Cref{prop-M-map}. After composing with the natural map of
  $M(\Cst(G))$ into $M(\Cst_\red(G))$, we obtain a homomorphism
  $V_\red \colon C_b(G\z) \to M(\Cst_\red(G))$ acting on
  $C_c(G)\subset \Cst_\red(G)$ exactly as at \Cref{eq-M-map}.  If $G$
  is \'etale, then the restrictions of both $V$ and $V_\red$ to
  $C_0(G\z)$ are obtained by inclusion and extension by zero of
  functions with compact support on $G\z$.
\end{remark}

\begin{thm}\label{main thm}
  Let $G$ be an \'etale groupoid with continuously varying isotropy
  subgroups that are uniformly subhomogeneous.  Then
  \[
    \dim\plusone_\nuc(\Cst_\red(G))\leq\sup_{u\in
      G\z}\dim\plusone\big(\Prim\Cst(G(u))\big)
    \dim\plusone(G\z)\dad\plusone(\gmiso).
  \]
\end{thm}

\begin{proof} 
  The open and continuous quotient map $\rho\colon G\to G/\Iso(G)$
  restricts to a homeomorphism of the unit spaces, and so we identify
  them and write $G\z$ for both.  We may assume that
  $\dim(G\z)=N<\infty$ and that $\dad(\gmiso)=d<\infty$.

  First assume that the unit space $G\z $ is compact. Let $\Ff$ be a
  finite subset of $C_c(G)\setminus\{0\}$ and let $\epsilon>0$. There
  exists a compact subset $K$ of $G$ such that $f\in \Ff$ implies
  $\supp f\subset K$. Since $K\inv$ is compact, we may assume that
  $K=K\inv $. Since $G\z $ is compact we may also assume that
  $G\z \subset K$.

  Since $G\z$ is compact, $\Cst_\red(G)$ is unital. 
  Therefore
  $V_{r}$ maps into $\cs_{r}(G)$ and we can  let
  $V_r \colon C(G\z) \to \Cst_\red(G)$ be (restriction of) the map
  from \Cref{rem V_r}.  Let $W\subset G$ be an open, symmetric and
  precompact neighborhood of $K$. Then $\rho(W)$ is an open precompact
  neighbourhood of the compact set $\rho(K)$ in $\gmiso$.  By applying
  \cite[Lemma~8.20]{GWY:2017:DAD}
  % (see also \cite[Remark~4.6]{CDaHGV})
  to
  $\epsilon \bigl((d+1)\max_{f\in
    \mathcal{F}}\|f\|_{\Cst_{\red}(G)}\bigr)^{-1}$ and $\rho(K)$,
  there exists $\delta>0$ such that for $h\in C(G\z)^+$ and $f\in \Ff$
  \begin{align}\label{eq delta}
    \sup_{\gamma\in \rho(K)}|h(s(\gamma))- h(r((\gamma))|<
    \delta\quad\Longrightarrow\quad
    \|fV_\red(h)-V_\red(h)f\|_\red<\frac{\epsilon}{d+1}\|h\|. 
  \end{align} 

  By assumption, $\gmiso$ has dynamic asymptotic dimension $d$, and
  applying \cite[Proposition~7.1]{GWY:2017:DAD} to $\delta$ and
  $\rho(W)$ gives open sets $U_0, \dots, U_d$ covering $G\z$ such that
  for $0\leq i\leq d$
  \begin{enumerate}[label=\textup{(\arabic*)}]
  \item\label{consequence1} the subgroupoids $H_i$ of $\gmiso$
    generated by
    $\set{\gamma\in \rho(W): s (\gamma), r (\gamma)\in U_i}$ are open
    and precompact;
  \item\label{consequence2} there exist continuous
    $ h_i \colon G\z \to[0,1]$ with support in $U_i$ such that
    $\sum_{i=0}^d h_{i}^2=1$, and
    \[\sup_{\gamma\in \rho(W)}|h_{i}(s(\gamma))- h_{i}(r((\gamma))|<
      \delta .\]
  \end{enumerate}

  Since $\gmiso$ is \'etale and each $H_i$ is precompact in $\gmiso$,
  it follows from the proof of \cite[Proposition~4.3(1)]{CDaHGV} that
  the size of the orbits in each $H_i$ is uniformly bounded. Thus the
  orbits in $\rho\inv(H_i)$ are uniformaly bounded as well. Since the
  isotropy subgroups of $G$ are uniformly subhomogeneous, it follows
  from \Cref{subhomg-characterisation} that each $\rho^{-1}(H_i)$ is a
  subhomogeneous groupoid.  Now we can replace \ref{consequence1} and
  \ref{consequence2} above: we have open sets $U_0, \dots, U_d$
  covering $G\z$ such that
  \begin{enumerate}[label=\textup{(\arabic*$'$)}]
  \item\label{consequence1'} the subgroupoids $\rho^{-1}(H_i)$ of $G$
    generated by $\{\gamma\in W: s (\gamma), r (\gamma)\in U_i\}$ are
    open and subhomogeneous;
  \item\label{consequence2'} there are continuous functions
    $h_i \colon G\z \to[0,1]$ with support in $U_i$ such that
    $\sum_{i=0}^d h_{i}^2=1$, and
    \[\sup_{\gamma\in W}|h_{i}(s(\gamma))- h_{i}(r(\gamma))|< \delta
      .\]
  \end{enumerate}
  Since $H_i$ is an open subgroupoid of $\gmiso$, $\rho^{-1}(H_i)$ is
  an open subgroupoid of $G$, and we may identify
  $\Cst_\red(\rho^{-1}(H_i))$ with a $\Cst$-subalgebra $B_i$ of
  $\Cst_\red(G)$.  Each $\rho\inv(H_i)$ is subhomogeneous and hence
  amenable by \Cref{subhomg-characterisation}. Thus
  $\Cst_\red(\rho\inv(H_i))=\Cst(\rho\inv(H_i))$ and
  \Cref{nuclear-dim-for-subhomog} gives
  \begin{align}
    \dim_\nuc\plusone(\Cst_\red(\rho\inv(H_i)))
    &\leq \dim\plusone(U_i) \sup_{u\in
      U_i}\dim\plusone(\Prim\Cst(G(u))),\notag\\ 
    \intertext{which, since $U_i$ is a subset of the metrisable space
    $G\z$, is} 
    &\leq \dim\plusone(G\z)\sup_{u\in
      G\z}\dim\plusone(\Prim\Cst(G(u))).
      \label{eq nuc dim of B_i}
  \end{align}
  We now show that $V_\red(h_i)$ and $B_i$ satisfy the hypotheses of
  \Cref{nuclear dim via sublagebras} with regard to the set
  $\Ff\subset X\coloneqq C_c(G)\setminus\{0\}$ and the $\epsilon>0$
  fixed above.

  Let $f\in \Ff$. Then
  $0\neq (V_\red(h_i)fV_\red(h_i)^*)(\gamma)=h_{i}(r(
  \gamma))f(\gamma)h_{i}(s(\gamma))$ implies that $\gamma\in K$ and
  $r(\gamma), s(\gamma)\in U_i$, that is,
  $V_\red(h_i)\Ff V_\red(h_i)^*\subset B_i$.  Since
  $\sum h_i^2=1_{C(\go)}$ we have
  \begin{align*}
    f+\sum_{i=1}^d V_\red(h_i)&\big(fV_\red(h_i)-V_\red(h_i)f \big)\\
                              &=f+\sum_{i=1}^d\big(V_\red(h_i)fV_\red(h_i)-V_\red(h_i)^2f \big)\\
                              &=\sum_{i=1}^d V_\red(h_i)fV_\red(h_i).
  \end{align*}
  Thus
  \begin{align*}
    \Big\|f-\sum_{i=1}^d V_\red(h_i)fV_\red(h_i)^*\Big\|_\red
    &=
      \Big\|\sum_{i=1}^d V_\red(h_i)\big(fV_\red(h_i)-V_\red(h_i)f  \big) \Big\|_\red\\
    &\leq 
      (d+1)\|h_i\|_\infty\|fV_\red(h_i)-V_\red(h_i)f \|_\red<\epsilon
  \end{align*}
  since $\|h_i\|\leq 1$ for $0\leq i\leq d$ and by the choice of
  $\delta$ in \Cref{eq delta} above.

  It now follows from \Cref{nuclear dim via sublagebras} that if the
  unit space of $G$ is compact, then the nuclear dimension of
  $\Cst_\red(G)$ is at most
  \[\sup_{u\in
      G\z}\dim\plusone\big(\Prim\Cst(G(u))\big)
    \dim\plusone(G\z)\dad\plusone(G/\Iso(G))-1,
  \]
  as needed

  Now suppose that $G\z $ is not compact.  Let $\cG$ be the
  Alexandrov groupoid as described in Section~\ref{sec:preliminaries}.
  By \Cref{dad quotient}, $\cG/\Iso(\cG)$ is isomorphic to
  $\alex(G/\Iso(G))$, and hence $\dad(\cG/\Iso(\cG))=\dad(G/\Iso(G))$
  using \cite[Proposition~3.13]{CDaHGV}.  Furthermore, by
  \cite[Lemma~2.6]{CDaHGV} we have $\dim (\cG\z )=\dim(G\z )$.  Now
  the compact case above gives the result for $\Cst_\red(\cG)$.  By
  \cite[Lemma~3.8]{CDaHGV}, $\Cst_\red(\cG)$ is the minimal
  unitization of $\Cst_\red(G)$, and hence it follows from
  \cite[Remark~2.11]{WinterZacharias:nuclear} that the nuclear
  dimension of $\Cst_\red(G)$ has the same bound.
\end{proof}

We obtain two corollaries:

\begin{cor}\label{main thm cor1}
  Let $G$ be an \'etale groupoid with continuously varying isotropy
  subgroups.
  \begin{enumerate}
  \item\label{main thm cor1-item1} Suppose that the isotropy subgroups
    are isomorphic and homeomorphic to a subgroup of an abelian group
    $S$.  Write $\hat S$ for the dual group of $S$.  Then
    \[
      \dim\plusone_\nuc(\Cst_\red(G))\leq\dim\plusone(\hat
      S)\dim\plusone(G\z)\dad\plusone(\gmiso).
    \]
  \item\label{main thm cor1-item2} Suppose that the isotropy subgroups
    are compact and uniformly subhomogeneous. Then
    \[
      \dim_\nuc\plusone(\Cst_\red(G))\leq
      \dim\plusone(G\z)\dad\plusone(\gmiso).
    \]
  \end{enumerate}
\end{cor}

\begin{proof} Using
  \Cref{nuclear-dim-for-subhomog-abelian-or-compact-isotropy}, in
  \Cref{eq nuc dim of B_i} in the proof of Theorem~\ref{main thm} we
  now have
  $ \dim\plusone_\nuc(\Cst_\red(\rho\inv(H_i)) \leq
  \dim\plusone(G\z)\dim\plusone(\hat S)$ in \cref{main thm
    cor1-item1}, and
  $\dim\plusone_\nuc(\Cst_\red(\rho\inv(H_i)) \leq\dim\plusone(G\z)$
  in \cref{main thm cor1-item2}.  The result then follows as in the
  proof of Theorem~\ref{main thm}.
\end{proof}

\begin{cor}\label{main thm cor2}
  Let $G$ be an \'etale groupoid with continuously varying isotropy
  subgroups that are uniformly subhomogeneous.  Then $G$ must be
  amenable if $\dim(G\z)$, $\dad(\gmiso)$ and
  $\sup_{u\in G\z}\dim\plusone\big(\Prim\Cst(G(u))\big)$ are finite.
\end{cor}

\begin{proof} By \Cref{main thm}, $\Cst_\red(G)$ has finite nuclear
  dimension and hence is nuclear by
  \cite[Remark~2.2]{WinterZacharias:nuclear}. Since $G$ is \'etale,
  this implies that $G$ is amenable
  \cite[Corollary~6.2.14]{AnRe:Am-gps}.
\end{proof}

\section{Application to \texorpdfstring{\cs}{C*}-algebras of directed
  graphs}
\label{sec directed graphs}

In this section we prove that all stably finite $\Cst$-algebras of
directed graphs havenuclear dimension is at most $1$.

\begin{thm}\label{application to AFE graphs}
  Let $E$ be a row-finite directed graph with no sources. Suppose that
  no return path in $E$ has an entrance. Then
  $\dim_\nuc(\Cst(E))\leq 1$.  In particular, all graph
  $\Cst$-algebras that are stably finite have nuclear dimension at
  most $1$.
\end{thm}

Before giving the proof, we require some preliminaries.  We will
construct a directed graph $F=F(E)$ from $E$ such that the graph
groupoid $G_F$ is principal, and show in \Cref{graph groupoid
  quotient} that if the isotropy subgroups of $G_E$ vary continuously,
then the quotient $G_E/\Iso(G_E)$ is isomorphic to the restriction of
$G_F$ to an open set $U\subset F^\infty$.  We will then apply
\Cref{main thm cor1} to conclude that $\dim_\nuc(\Cst(G_E))\leq 1$.
The motivation for our construction comes from \cite[\S7,
Example~1]{Clark-anHuef-RepTh} which was also used in
\cite[Example~5.8]{CDaHGV}.

We start with the required background on directed graphs.  Let
$E=(E^0,E^1,r,s)$ be a directed graph.  We use the convention from
\cite{Iain:CBMS} and list paths from the range.  A finite path $\mu$
of length $|\mu|=k$ is a sequence $\mu=\mu_{1}\mu_{2}\cdots\mu_k$ of
edges $\mu_i\in E^1$ with $s(\mu_j)=r(\mu_{j+1})$ for $1\le j\le k-1$;
we extend the source and range to finite paths by $s(\mu)=s(\mu_k)$
and $r(\mu)=r(\mu_{1})$.  An infinite path $x=x_{1}x_{2}\cdots$ is
similarly defined; note that $s(x)$ is undefined.  We think of the
vertices in $E^0$ as finite paths of length $0$. Let $E^*$ and
$E^\infty$ denote the sets of finite and infinite paths in $E$,
respectively. If $\mu=\mu_{1}\cdots\mu_k$ and $\nu=\nu_{1}\cdots\nu_j$
are finite paths with $s(\mu)=r(\nu)$, then $\mu\nu$ is the path
$\mu_{1}\cdots\mu_k\nu_{1}\cdots\nu_j$; when $x\in E^\infty$ with
$s(\mu)=r(x)$, define $\mu x$ similarly.

A \emph{return path} is a finite path
$\mu=\mu_{1}\mu_{2}\cdots\mu_{|\mu|}$ of non-zero length such that
$s(\mu)=r(\mu)$; it is \emph{simple} if
$\{r(\mu_1), \dots, r(\mu_{|\mu|})\}$ are distinct. A return path
$\mu$ has an \emph{entrance} if $|r\inv(r(\mu_i))|>1$ for some
$1\leq i\leq |\mu|$ and has an \emph{exit} if $|s\inv(s(\mu_i))|>1$
for some $1\leq i\leq |\mu|$.

The cylinder sets
\[
  Z(\mu)\coloneqq\{x\in E^\infty:
  x_{1}=\mu_{1},\ldots,x_{|\mu|}=\mu_{|\mu|}\},
\]
parameterised by $\mu\in E^*$, form a basis of compact, open sets for
a locally compact, totally disconnected, Hausdorff topology on
$E^\infty$ by \cite[Corollary~2.2]{KumjianPaskRaeburnRenault}.  We say
$E$ is row-finite if $r\inv (v)$ is finite for every $v\in E^0$. A
vertex $v\in E^0$ is called a source if $s\inv(v)=\emptyset$.

Given a row-finite graph $E$ with no sources, the graph groupoid $G_E$
is defined in \cite{KumjianPaskRaeburnRenault} as follows. Two paths
$x,y\in E^\infty$ are shift equivalent with lag $k\in\ZZ$ (written
$x\sim_k y$) if there exists $N\in\NN$ such that $x_i=y_{i+k}$ for all
$i\ge N$. Then the groupoid is
\[ G_E\coloneqq\{(x,k,y)\in E^\infty\times \ZZ\times E^\infty :
  x\sim_k y\}.\] with composable pairs
\[
  G_E\comp \coloneqq\{\big((x,k,y),(y,l,z)\big):(x,k,y),(y,l,z)\in
  G_E\},
\]
and composition and inverse given by
\[
  (x,k,y)\cdot (y,l,z)\coloneqq
  (x,k+l,z)\quad\text{and}\quad(x,k,y)\inv \coloneqq (y,-k,x).
\]
For $\mu,\nu\in E^*$ with $s(\mu)=s(\nu)$, set
\[
  Z(\mu,\nu)\coloneqq\{(x,k,y) : x\in Z(\mu), y\in Z(\nu),
  k=|\nu|-|\mu|, x_i=y_{i+k}\text{ for } i>|\mu|\}.
\]
By \cite[Proposition~2.6]{KumjianPaskRaeburnRenault},
$ \{Z(\mu,\nu) : \mu,\nu\in E^*, s(\mu)=s(\nu)\} $ is a basis of
compact, open sets for a second-countable, locally compact, Hausdorff
topology on $G_E$ such that $G_E$ is an \'etale groupoid; after
identifying $(x,0,x)\in G_E\z $ with $x\in E^\infty$, the subspace
topology on $G_E\z $ coincides with the topology on $E^\infty$.

\begin{remark}
  If no return path in $E$ has an entrance, then every return path is
  a simple return path.  By \cite{Schafhauser-AFE}, $\Cst(E)$ is
  stably finite (equivalently, AF-embeddable) if and only if no return
  path in $E$ has an entrance.
\end{remark}

For the remainder of this section $E$ will be a row-finite directed
graph with no sources such that \emph{no return path has an entrance}.
We also assume that every return path in $E$ has an exit (for
otherwise, the return path is disconnected from the rest of the
graph).

The next lemma shows that \Cref{main thm} applies to the graph
groupoid $G_E$ because its isotropy subgroupoid is open whence the
isotropy subgroups vary continuously by \Cref{lem hypotheses for lc
  quotient} (since uniform subhomogeneity is automatics as the
isotropy is abelian).

\begin{lemma}\label{lem open isotropy} The isotropy subgroupoid
  $\Iso(G_E)$ of $G_E$ is open.
\end{lemma}

\begin{proof}
  Fix $(x,k,x)\in \Iso(G_E)$. If $k=0$, then $(x,0,x)\in G_E\z$ which
  is an open set of $G_E$ contained in $\Iso(G_E)$. Next, suppose that
  $k\neq 0$; we may assume that $k>0$.  Then $x\sim_k x$, that is,
  there exists $N\in\NN$ such that $x_i=x_{i+k}$ for $i\geq N$. Thus
  $\alpha=x_i\cdots x_{i+k}$ is a return path.  Let
  $\mu=x_1\cdots x_i$ and $\nu=x_1\dots x_{i+k}$. Then
  $x\in Z(\mu, \nu)$ which is open in $G$. Further, since $\alpha$ has
  no entry, $Z(\mu, \nu)\subset \Iso(G_E)$. Thus $\Iso(G_E)$ is open.
\end{proof}

To distinguish range and source maps of the graph and of the graph
groupoid, we write $r_E$ and $s_E$ for the range and source in $E$.
For better \emph{and} worse, we have chosen our notation to follow
that of \cite[\S7, Example~1]{Clark-anHuef-RepTh}, and so we now write
return paths $\alpha$ as $\alpha_0\dots \alpha_{|\alpha|-1}$ where
$|\alpha|$ is the length of $\alpha$ (sorry).

We declare two return paths to be \emph{equivalent} if their sets of
vertices are the same.  We choose a set of representatives $R$ for the
equivalence classes of return paths such that if
$\alpha=\alpha_0\alpha_1\dots \alpha_n\in R$, then there is an exit to
$\alpha$ with source $r_E(\alpha)=r_E(\alpha_0)$.

We let $R^0$ be the vertices on return paths in $R$ and set
$V^0=E^0\setminus R^0$. For each
$\alpha=\alpha_0\alpha_1\dots \alpha_{|\alpha|-1}\in R$ we label its
vertices such that
\[
  w_{\alpha, i}=r_E(\alpha_i) \quad\text{for
    $0\leq i\leq |\alpha|-1$.}
\]
Let $F=(F^0, F^1, r, s)$ be the directed graph where
\begin{equation*}\label{F^0}
  F^0\coloneqq \set{v'\colon v\in V^0}\bigcup_{\alpha\in R}
  \set{w'_{\alpha, i}\colon i\in\NN} 
\end{equation*}
(thus every vertex of $E$ is duplicated and infinitely many vertices
$w'_{\alpha,i}\ (i\geq |\alpha|)$ are added for each $\alpha\in R$)
\begin{equation*}\label{F^1}
  F^1\coloneqq \set{e'\colon \text{$e\in E^1$ with $r_E(e)\in V^0$}}
  \bigcup_{\alpha\in R}\set{\alpha_i'\colon i\in\NN} 
\end{equation*}
and setting
\begin{align*}\label{range and source in F}
  &r_F(e')=r_E(e)'\text{\ and\ }s_F(e')=s_E(e)'\quad\text{if\ }
    r_E(e)\in V^0;\\
  &r_F(\alpha_i')=w_{\alpha, i}'\text{\ and\
    }s_F(\alpha_i')=w_{\alpha, i+1}'\quad\text{for\ }i\in\NN\notag 
\end{align*}
(thus edges not on return paths in $E$ are duplicated, return paths in
$E$ are ``unfurled'' and an infinite tail is added to them).  By
construction, $F$ is a directed graph with no return paths which we
call the unfurled version of $E$.

Allowing paths to be the empty path or an appropriate vertex, every
finite path of $E$ is of the form
\[
  \mu e \alpha^n\beta
\]
where $\mu$ is a path with range and source in $V^0$, $e$ is an exit
to a return path $\alpha$ (not necessarily in $R$) and $\beta$ is a
subpath of $\alpha$.  If $\alpha$ has range $w_{\alpha,i}$, then
$\mu e \alpha^n\beta$ corresponds to the unique path
$\mu'e'p_{\alpha, n, \beta}$ where $p_{\alpha, n, \beta}$ is the
unique path in $F$ with range $w_{\alpha_i}'$ of length
$n(|\alpha|-1)+|\beta|$. This gives an injection
$\Phi\colon E^*\to F^*$ with range in the set of paths of $F$ with
ranges in
\begin{equation}\label{U^0}
  U^{0}\coloneqq \set{v'\colon v\in V^0}\bigcup_{\alpha\in
    R}\set{w_{\alpha, i}'\colon 0\leq i\leq |\alpha|-1}. 
\end{equation}

An infinite path in $E^\infty$ is either of the form $x$ where the
vertices of $x$ are all in $V^0$ or of the form $\mu e\alpha^\infty$
where $\mu$ is a finite path in $E^*$ with range and source in $V^0$
and $e$ is an exit to a return path $\alpha$; the natural extension
$\Psi\colon E^\infty \to F^\infty$ of $\Phi$ is given by $\Psi(x)=x'$
and $\Psi(\mu e\alpha^\infty)$ is the unique infinite path in $F$
extending $\mu'e'$.

\begin{lemma}\label{rho-welldefined} Let $x, y\in E^\infty$ such that
  $x\sim_k y$. Then there exists unique $l=l(x,k,y)$ such that
  $\Psi(x)\sim_l\Psi(y)$.
\end{lemma}

\begin{proof} There are two cases:
  \begin{enumerate}
  \item all the vertices on $x$ are in $V^0$ (and then the same is
    true of $y$);
  \item $x$ contains a return path $\alpha$ (and then so does $y$).
  \end{enumerate}
  First, suppose that all vertices of $x$ and $y$ are in $V^0$. Then
  $\Psi(x)\sim_k\Psi(y)$, so $l=k$.  Second, suppose that
  $x=\mu e\alpha^\infty$ and $y=\nu f\beta\alpha^\infty$ where $e, f$
  are exits from a return path $\alpha$ and $\beta$ is a subpath of
  $\alpha$. Then $\Psi(x)\sim_l\Psi(y)$ where
  $l=|\mu|+1-(|\nu|+1+|\beta|)$. This establishes existence.

  Suppose that $\Psi(x)\sim_{l_1}\Psi(y)$ and
  $\Psi(x)\sim_{l_2}\Psi(y)$. Then $(\Psi(x),l_1,\Psi(y))\in G_F$ and
  $(\Psi(x),l_2,\Psi(y))\in G_F$, giving
  \[(\Psi(x), l_1-l_2,
    \Psi(x))=\big(\Psi(x),l_1,\Psi(y))(\Psi(x),l_2,\Psi(y)\big)^{-1}\in
    G_F.\] Since $F$ has no return paths, $G_F$ is principal by
  \cite[Proposition~8.1]{Hazlewood-anHuef}. Thus $l_1=l_2$.
\end{proof}

\begin{thm} \label{graph groupoid quotient} Let $E$ be a row-finite
  directed graph with no sources such that no return path has an
  entrance.  Let $F$ be the unfurled version of $E$ described above,
  and set
  \begin{equation*}\label{set U}
    U=\bigcup_{\{v'\colon v\in V^0\}} Z(v')
    \bigcup_{\stackrel{\alpha\in R}{0\leq i\leq |\alpha|-1}}
    Z(w'_{\alpha, i}).
  \end{equation*}
  Then $U$ is open and $\Upsilon\colon G_E\to (G_F)|_U$ defined by
  \begin{equation*}\label{define rho}
    \Upsilon((x,k,y))=(\Psi(x), l(x,k,y), \Psi(y)).
  \end{equation*}
  is a surjective, open and continuous homomorphism that factors
  through an isomorphism
  $\tilde \Upsilon\colon G_E/\Iso(G_E)\to G_F|_U$ of topological
  groupoids.
\end{thm}

\begin{proof}
  That $\Upsilon$ is well-defined follows because $l(x,k,y)$ is unique
  by \Cref{rho-welldefined} and because $\Psi$ has range $U$; indeed
  $\Upsilon$ is onto $(G_F)|_U$.  Notice that $U$ is open because it
  is a union of cylinder sets.

  Let $\big((x, j, y), (y, k, z)\big)\in G_E\comp$.  Then
  \begin{align*}
    \Upsilon\big((x, j, y) (y, k, z)\big)
    &=\Upsilon\big((x, j+k, z) \big)
      =(\Psi(x), l_{j+k}, \Psi(z))
      \intertext{and}
      \Upsilon \big((x, j, y)\big)\Upsilon\big( (y, k, z)\big)
    &=(\Psi(x), l_{j}, \Psi(y))(\Psi(y), l_{k}, \Psi(z))=(\Psi(x),
      l_{j}+l_{k}, \Psi(z)). 
  \end{align*}
  Since $F$ has no return paths, $G_F$ is principal by
  \cite[Proposition~8.1]{Hazlewood-anHuef}, which implies that
  $l_{j}+l_{k}=l_{j+k}$. Thus $\Upsilon$ is an algebraic homomorphism.

  Let $j,k\in\ZZ$ such that $(x, k, y), (x, j, y)\in G_E$.  Then since
  $\Upsilon$ is a homomorphism
  \[\Upsilon((x,k, y))\Upsilon((x,j,y))^{-1}
    =\Upsilon((x, k-j, x))=(\Psi(x), 0, \Psi(x))
  \]
  because $G_F$ is principal. Thus
  $\Upsilon((x, k, y))=\Upsilon((x,j,y))$. Since $\Psi$ is injective
  it follows that
  \[
    \Upsilon((x,k, y))=\Upsilon((w,j,z))\Longleftrightarrow
    (x,k,y)(w,j,z)^{-1}=(x, k-j,x)\in \Iso(G_E).
  \]
  Thus $\Upsilon$ factors through an algebraic isomorphism
  $\tilde \Upsilon\colon G_E/\Iso(G_E)\to G_F|_U$, that is,
  $\Upsilon= \tilde\Upsilon\circ \rho$ where
  $\rho\colon G_E\to G_E/\Iso(G_E)$ is the quotient map.

  Let $\Phi\colon E^*\to F^*$ be the map described above at
  \Cref{U^0}. Then every non-empty, basic open neighbourhood of
  $(G_F)|_U$ is of the form $Z(\xi',\eta')$ where
  $\xi',\eta'\in\range\Phi$, say $\xi'=\Phi(\xi)$ and
  $\eta'=\Phi(\eta)$ for $\xi, \eta\in E^*$.  Then
  \begin{align*}
    \Upsilon^{-1}\big(Z(\xi',\eta')\big)
    &=
      \Upsilon^{-1}\big(\{(\xi'z, |\xi'|-|\eta'|, \eta' z)\colon z\in
      U \} \big)\\ 
    &=
      \{(\xi x, k, \eta x)\in G_E\colon x\in E^\infty,  k\in\ZZ \}\\
    &=Z(\xi,\eta)\Iso(G_E).
  \end{align*}
  This also gives
  \[
    Z(\xi',\eta')=\Upsilon\big(\Upsilon^{-1}\big(Z(\xi',\eta')\big)
    =\Upsilon\big(Z(\xi,\eta)\Iso(G_E)\big)=\Upsilon\big(Z(\xi,
    \eta)\big).
  \]
  Thus $\Upsilon$ maps basic open neighbourhood to open
  neighbourhoods, and hence $\Upsilon$ is open.  If $U$ is open in
  $G_E/\Iso(G_E)$, then $\tilde\Upsilon(U)=\Upsilon(q^{-1}(U))$ is
  open, and so $\tilde\Upsilon$ is also open.

  Finally, the isotropy subgroupoid $\Iso(G_E)$ is open in $G_E$ by
  \Cref{lem open isotropy}. Since products of open sets in a groupoid
  are open,
  $\Upsilon^{-1}\big(Z(\xi',\eta')\big) =Z(\xi,\eta)\Iso(G_E)$ is
  open. Thus inverse images under $\Upsilon$ of basic open
  neighbourhoods are open, and hence $\Upsilon$ is continuous. Thus
  $\tilde\Upsilon$ is an isomorphism of topological groupoids.
\end{proof}

% \begin{remark} THIS REMARK IS FALSE There are directed graphs where
%   no return paths have entries but the isotropy subgroups do not
%   vary continuously. For example, take the graph $E$ of
%   \cite[Example~1 in \S7]{Clark-anHuef-RepTh} and change the return
%   paths of length $1,2, \dots$ to all have length $1$. Let $x$ be
%   the unique infinite path with vertices $v_n$ for $n\in\NN$, and
%   for each $n\in\NN$ let $x_n$ be any infinite path with range $v_0$
%   going through $w_{n,0}$. Then $x_n\to x$ in $E^\infty$ as
%   $n\to\infty$. But the isotropy subgroup at $x_n$ and $x$ are
%   $G_E(x_n)=\{(x_n, k, x_n)\colon k\in\ZZ\}$ and
%   $G_E(x)=\{(x,0,x)\}$, respectively, and so $G_E(x_n)$ does not
%   converge to $G_E(x)$ as $n\to\infty$ in the Fell topology.
% \end{remark}

\begin{proof}[Proof of \Cref{application to AFE graphs}]
  We identify $\Cst(E)$ and $\Cst(G_E)$. By
  \cite[Proposition~2.6]{KumjianPaskRaeburnRenault}, the unit space
  $E^\infty$ of $G_E$ is second countable, locally compact, Hausdorff
  (hence metrisable) and has a basis for a topology consisting of
  compact and open (hence clopen) sets. Thus $\dim(E^\infty)=0$ by
  \cite[Theorem~2.8.1]{Coornaert}. Since the isotropy subgroupoid is
  open by \Cref{lem open isotropy}, $G_E/\Iso(G_E)$ is locally
  compact, Hausdorff and \'etale by \Cref{lem hypotheses for lc
    quotient}.  All the isotropy subgroups are homeomorphic to $\ZZ$
  and the dual group $\TT$ of $\ZZ$ has $\dim(\TT)=1$
  \cite[\S1.5.9]{Engelking}.  Since no return path in $E$ has an
  entrance,
  % and since the isotropy subgroups vary continuously,
  by \Cref{graph groupoid quotient} there is a graph $F$ and an open
  neighbourhood $U$ of $F^\infty$ such that $G_E/\Iso(G_E)$ is
  isomorphic to the restriction of $G_F$ to $U$. Since $F$ has no
  return paths, $\dad(G_F)=0$ by \cite[Lemma~5.7]{CDaHGV}. Further,
  since $U$ is open, $\dad((G_F)|_U)=0$ by
  \cite[Lemma~3.11]{CDaHGV}. By \Cref{main thm cor1},
  \[\dim_\nuc(\Cst(G_E))=\dim_\nuc(\Cst(G_F|_U))\leq
    (1+1)(0+1)(0+1)-1=1.\qedhere\]
\end{proof}

The estimate found in \Cref{application to AFE graphs} is sharp as can
be seen from the graph $E$ in \cite[Example~5.8]{CDaHGV}. There
$\Cst(E)$ is stably finite (because $E$ has no return paths with
entries) but not AF (because $E$ has return paths); the computation in
\cite{CDaHGV} and \Cref{application to AFE graphs} both give
$\dim_\nuc(\Cst(E))=1$.

\section{Nuclear dimension of \texorpdfstring{\cs}{C*}-algebras of
  twists}
\label{sec:nuclear-dimension-cs-1}

In this section we demonstrate how to use our techniques to bound the
nuclear dimension of a $\Cst$-algebra of a non-etale groupoid by
considering a twist over an \'etale groupoid.

\begin{thm}\label{thm-twist-and-recover} 
  Let $G$ be an \'etale groupoid and let $(\Sigma,\iota,\pi)$ be a
  twist over $G$.  Then
  \[\dim_\nuc\plusone(\Cst_\red(\Sigma))\leq
    \dad\plusone(G)\dim\plusone(G\z)
  \]
\end{thm}

Before giving the proof, we require some preliminaries.
Since the twisted groupoid $\Cst$-algebra is a quotient of
$\Cst_\red(\Sigma)$ we recover
\cite[Theorem~3.2]{Bonicke-Li-Nuclear-dim} from
\Cref{thm-twist-and-recover}.  Conversely,
\cite[Theorem~3.2]{Bonicke-Li-Nuclear-dim} and the decomposition from
\cite[Proposition~3.7]{ikrsw:nzjm21}---where the $A$ in that result is
the
circle $\mathbf T$---of
$\Cst_\red(\Sigma)$ into a direct sum of twisted groupoid
$\Cst$-algebras can be used to prove \Cref{thm-twist-and-recover}. But
here we want to test our techniques for non-\'etale groupoids.

Even if $G\z=\Sigma\z$ is compact, $\Cst_\red(\Sigma)$ is not unital,
and we need to develop a non-unital version of \Cref{nuclear dim via
  sublagebras}; we do this in \Cref{abstraction-nonunital} below after
considering the nuclear dimension of extensions of subhomogeneous
$\Cst$-algebras by subhomogeneous $\Cst$-algebras.

\begin{lemma}\label{subhomg extension}  Let $I$ be an ideal of a $\Cst$-algebra $A$.
  If $I$ is $m$-subhomogeneous and $A/I$ is $n$-subhomogenous, then
  $A$ is $\max\{m,n\}$-subhomogeneous. If $A$ is separable, then
  $\dim_\nuc{A}=\max\{\dim_\nuc(I), \dim_\nuc(A/I)\}$.
\end{lemma}
\begin{proof} Let $q\colon A\to A/I$ be the quotient map.  Fix an
  irreducible representation $\pi\colon A\to B(H_\pi)$.  If
  $\pi|_I=0$, then $\pi\circ q\colon A/I\to B(H_\pi)$ is an
  irreducible representation of $A/I$; since $A/I$ is
  $n$-subhomogeneous, $\dim(H_\pi)\leq n$.  If $\pi|_I\neq 0$, then
  $\pi|_I\colon I\to B(H_\pi)$ is an irreducible representation of $I$
  (see, for example, the proof of \cite[Proposition~A26]{tfb}), and
  hence $\dim(H_\pi)\leq m$. Thus $A$ is $\max\{m,n\}$-subhomogeneous.

  Suppose that $A$ is separable. Then $\Prim A $ is second
  countable. Also, by \cite[\S1.6]{Winter:decomposition-rank} we have
  \[
    \dim_\nuc(A)= \max_k\{\dim(\Prim_k A)\}.
  \]
  Fix $k$ such that $\Prim_k A \neq\emptyset$. By
  \cite[Proposition~3.6.4]{Dixmier}, $\Prim_k A $ is locally compact
  and Hausdorff, and hence $\Prim_k A $ is metrisable.  It follows
  from \cite[Proposition~A26]{tfb} that
  $\{P\in \Prim_k A :I\not\subset P\}$ is open in $\Prim_k A $ and is
  homeomorphic to $\Prim_k I$, and that
  $\{P\in \Prim_k A :I \subset P\}$ is closed in $\Prim_k A $ and is
  homeomorphic to $\Prim_k A/I$.  Now $\Prim_k A $ is a separable
  metric space which is the union of an open and a closed subset both
  of dimension at most $\max\{\dim_\nuc(I), \dim_\nuc(A/I)\}$. Thus
  \[\dim(\Prim_k A )\leq \max\{\dim_\nuc(I), \dim_\nuc(A/I)\}\] by
  \cite[Corollary~1.5.5]{Engelking}. It follows that
  $\dim_\nuc(A)\leq \max\{\dim_\nuc(I), \dim_\nuc(A/I)\}$.  The
  reverse inequality follows from
  \cite[Proposition~2.9]{WinterZacharias:nuclear}.
\end{proof}

\begin{prop}\label{abstraction-nonunital}
  Let $A$ be a non-unital $\Cst$-algebra, let $X\subset A$ be such
  that $\operatorname{span}(X)$ is dense in $A$ and let $M$ be a
  unital, commutative $\Cst$-subalgebra of $M(A)$. Let $d, n\in
  \NN$. Suppose that for every finite $\Ff\subset X$ and every
  $\epsilon>0$, there exist $\Cst$-subalgebras $B_0,\ldots ,B_d$ of
  $A$ and $b_0,\ldots ,b_d\in M$ such that $\sum_{i=0}^d b_ib_i^*=1$,
  and for $1\leq i\leq d$ we have $\dim_{\nuc}(B_i)\leq n$,
  $b_{i}\Ff b_{i}^*\in B_{i}$ and
  $\|x-\sum_{i=0}^d b_ixb_i^*\big\|<\epsilon$ for $x\in\Ff$.
  Then
  \[\dim\plusone_\nuc(A)\leq (d+1)\big(\dim(\widehat{M})+n+2\big).\]
  If the $B_i$ are separable and subhomogeneous for $0\leq i\leq d$,
  then
  \[\dim\plusone_\nuc(A)\leq (d+1)\big(\max\{n, \dim(\widehat M)
    \}+1\big).\]
\end{prop}

\begin{proof} 
  Let $C$ be the $\Cst$-algebra generated by $A$ and $M$.  Then
  $C=A+M$ and $A$ is an ideal of the unital $\Cst$-algebra $C$.  Since
  $\lsp(X)$ is dense in $A$, $\lsp(X\cup M)$ is dense in $C$.

  Fix $\epsilon>0$ and let $\Ff\subset X\cup M$ be a finite subset.
  Then $\Ff\cap X$ is a finite subset of $X$. By assumption, there
  exist $\Cst$-subalgebras $B_0,\ldots ,B_d$ of $A$ and
  $b_0,\ldots ,b_d\in M$ such that $\sum_{i=0}^d b_ib_i^*=1$, and for
  $1\leq i\leq d$ and $x\in \Ff\cap X$ we have
  $\dim_{\nuc}(B_i)\leq n$, $b_{i}xb_{i}^*\in B_{i}$ and
  $\big\|x-\sum_{i=0}^d b_ixb_i^*\big\|<\epsilon$.

  For each $i$, let $C_i$ be the $\Cst$-algebra generated by $B_i$ and
  $M$, so that $C_i=B_i+M$ and $B_i$ is an ideal in $C_i$.  By
  \cite[Proposition~2.9]{WinterZacharias:nuclear}
  \begin{equation}\label{eq-extension-lemma}
    \dim_\nuc(C_i)\leq \dim_\nuc(B_i)+\dim_\nuc(M)+1=n+\dim(\widehat M)+1
  \end{equation}
  since $M$ is commutative.  Each $b_i$ has norm at most $1$ in $C$
  and for all $y\in \Ff$ we have $b_iyb_i^*\in C_i$ and
  $\big\|y-\sum_{i=0}^d b_iyb_i^*\big\|<\epsilon$.  Now \Cref{nuclear
    dim via sublagebras} applied to $C$ and $X\cup M$ implies that
  \[\dim_\nuc(A)\leq \dim_\nuc(C)\leq (d+1)\big(n+\dim(\hat
    M)+2\big)-1.\] If all the $B_i$ are separable and subhomogeneous,
  then using \Cref{subhomg extension} at \Cref{eq-extension-lemma} we
  get $\dim_\nuc(C_i)= \max\{n, \dim(\widehat M)\}$; we then apply
  \Cref{nuclear dim via sublagebras} to get the better bound.
\end{proof}

\begin{lemma}\label{commutator with bisection2}
  Let $G$ be an \'etale groupoid and let $(\Sigma,\iota,\pi)$ be a
  twist over $G$.
  \begin{enumerate}
  \item\label{X} Set
    \[
      X\coloneqq\{f\in C_c(\Sigma): \pi(\supp(f))\text{\ is a
        bisection in $G$}\}
    \]
    Then $C_c(\Sigma)=\lsp(X)$.
  \item\label{commutator with bisection} Let $K$ be a compact subset
    of $\Sigma$ and let $W$ be an open precompact neighbourhood of $K$
    in $\Sigma$. Let $V_\red\colon C_{0}(\Sigma\z)\to M(\cs_{r}(\Sigma))$
    be the map of \Cref{rem V_r} and let $h\in C_0(\Sigma\z)$. Then
    for all $f\in X$ with $\supp(f)\subset K$ we have
    \[
      \|V_\red(h)f-fV_\red(h)\|_\red\leq
      \sup_{\gamma\in\pi(W)}|h(r(\gamma))-h(s(\gamma))|\|f\|_\red.
    \]
  \end{enumerate}
\end{lemma}

\begin{proof}
  Fix $f\in C_c(\Sigma)$. Let $W$ be an open precompact neighbourhood
  of $\supp(f)$. Since $G$ is \'etale there is an open cover
  $\{U_i\}_{i=1}^n$ of $\pi(W)$ such that each $U_i$ is a bisection.
  Set $W_i\coloneqq W\cap\pi^{-1}(U_i)$ and let $\{f_i\}_{i=1}^n$ be a
  partition of unity of $W$ subordinate to $\{W_i\}_{i=1}^n$.  Extend
  each $f_i\colon W\to [0,1]$ to $f_i\colon \Sigma\to[0,1]$ by
  $f_i(e)=0$ if $e\notin W$. Then it is straightforward to check that
  for each $i$ the point-wise product $f\cdot f_i\in C_c(\Sigma)$ with
  $\pi(\supp(f\cdot f_i))\subset U_i$.  Since
  $f=f\cdot \sum_{i=1}^n f_i=\sum_{i=1}^n f\cdot f_i$. This proves
  \cref{X}.

  For \cref{commutator with bisection}, let $f\in X$ with
  $\supp(f)\subset K$.  Fix $u\in \Sigma\z $ and let
  $L^{u}
  %=\Ind\delta_{u}
  \colon C_c(\Sigma)\to
  B(L^{2}(\Sigma_{u},\lambda_{u}))$, so that
  \[
    \big(L^{u}(f)\xi\big)(e)
    =\int_{\Sigma}f(e_{1})\xi(e_{1}^{-1}e)\,\dd\lambda^{r(e)}(e_{1})
  \]
  for $\xi\in C_{c}(\Sigma_{u}) \subset L^{2}(\Sigma_{u},\lambda_{x})$
  and $e\in \Sigma_{u}$.  Then
  \begin{align*}
    \big(L^{u}
    &(V_\red(h)f-fV_\red(h))\xi\big)(e) 
    \\ 
    &=\sum_{\alpha\in r(e)G}\int_\TT\big(h(r(t\cdot\mathfrak{c}(
      \alpha)))-(h(s(t\cdot\mathfrak{c}(\alpha)))
      f(t\cdot\mathfrak{c}(\alpha))\xi(  
      t\inv \cdot\mathfrak{c}(\alpha)\inv e )\, \dd t\\ 
    &=\sum_{\alpha\in r(e)G} \big(h(r(\mathfrak{c}(
      \alpha)))-h(s(\mathfrak{c}(\alpha)))\big)\int_\TT
      f(t\cdot\mathfrak{c}(\alpha))\xi( t\inv
      \cdot\mathfrak{c}(\alpha)\inv e )\, \dd t. 
  \end{align*}
  Notice that $t\cdot\mathfrak{c}(\alpha)\in\supp(f)$ implies that
  $\alpha\pi(t\cdot\mathfrak{c}(\alpha))\in\pi(\supp(f))$ which is a
  bisection.  So there is at most one $\alpha_e\in\pi(\supp(f))$ with
  $r(\alpha)=r(e)$. Thus either
  \[\big(L^{u}(V_\red(h)f-fV_\red(h))\xi\big)(e)=0=\big(L^u(f)\xi\big)(e)\]
  or
  \begin{align*}
    \big(L^{u}(V_\red(h)
    &f-fV_\red(h))\xi\big)(e) \\
    &=\big(h(r(\mathfrak{c}(
      \alpha_e)))-h(s(\mathfrak{c}(\alpha_e)))\big)\int_\TT
      f(t\cdot\mathfrak{c}(\alpha_e))\xi( t\inv
      \cdot\mathfrak{c}(\alpha_e)\inv e )\, \dd t\\ 
    &=\big(h(r(\mathfrak{c}(
      \alpha_e)))-h(s(\mathfrak{c}(\alpha_e)))\big)\sum_{\alpha\in
      r(e)G}\int_\TT f(t\cdot\mathfrak{c}(\alpha))\xi( t\inv
      \cdot\mathfrak{c}(\alpha)\inv e )\, \dd t\\ 
    &=\big(  h(r(\alpha_e))-h(s( \alpha_e))\big)\big(L^u(f)\xi\big)(e).
  \end{align*}
  In either case,
  \begin{align*}
    \|L^{u}(V_\red(h)f-fV_\red(h))\xi\|^2
    &=\int_\Sigma |\big(L^{u}(V_\red(h)f-fV_\red (h))\xi\big)(e)|^2\, \dd\lambda_x(e)\\
    &\leq \sup_{\gamma\in \pi(W)}|h(r(\gamma))-h(s(
      \gamma))|^2\|L^u(f)\xi\|^2, 
  \end{align*}
  and the result follows.
\end{proof}

\begin{proof}[Proof of \Cref{thm-twist-and-recover}]
  The proof is very similar to that of \Cref{main thm}. There
  $\rho\colon G\to G/\Iso(G)$ was a quotient map onto the principal
  groupoid $G/\Iso(G)$ which had finite dynamic asymptotic
  dimension. Here $\pi\colon \Sigma\to G$ is the quotient map onto the
  not-necessarily principal $G$ which has finite dynamic asymptotic
  dimension. In both, the finite dynamic asymptotic dimension gives
  precompact subgroupoids which pull back to subhomogeneous
  subgroupoids whose $\Cst$-algebras have uniformly bounded nuclear
  dimension.

  We may assume that $\dim(G\z)=N$ and $\dad(G)=d$ are finite.  First
  assume that $\Sigma\z=G\z$ is compact.  By \Cref{commutator with
    bisection2},
  \[
    X=\set{f\in C_c(\Sigma): \text{$\pi(\supp(f))$ is a bisection in
        $G$}}
  \]
  has dense span in $\Cst(\Sigma)$.  Fix $\epsilon>0$ and a finite
  subset of $\Ff$ of $X$.  There exists a compact subset $K=K^{-1}$ of
  $\Sigma$ such that $\Sigma\z\subset K$ and $f\in \Ff$ implies
  $\supp f\subset K$.

  Let $W\subset \Sigma$ be an open, symmetric and precompact
  neighborhood of $K$. Then $\pi(W)$ is an open precompact
  neighbourhood in $G$. Since $\dad(G)=d$, applying
  \cite[Theorem~7.1]{GWY:2017:DAD} to the precompact, open subset
  $\pi(W)$ of $G$ with
  $\epsilon\bigl( (d+1)\max_{a\in \mathcal{F}}\|a\|\bigr)^{-1}$ gives
  open sets $U_0, \dots, U_d$ covering $G\z$ such that for
  $0\leq i\leq d$ the subgroupoids $H_i$ of $G$ generated by
  $\set{\gamma\in \pi(W): s (\gamma), r (\gamma)\in U_i}$ are open and
  precompact, and there exist $ h_i \in C(G\z)$ with support in $U_i$
  such that $0\leq h_i\leq 1$ and $\sum_{i=0}^d h_{i}^2=1_{C(\go)}$,
  and
  \begin{equation*}
    \sup_{\gamma\in \pi(W)}|h_{i}(s(\gamma))- h_{i}(r((\gamma))|<
    \epsilon\big((d+1)\max_{a\in \mathcal{F}}\|a\|\big)\inv. 
  \end{equation*}
  
  Since each $H_i$ is precompact in the \'etale groupoid $G$, there
  exists $M_i$ such that for all $u\in H_i\z$ we have
  $|(H_i)_u|\leq M_i$ (see proof of
  \cite[Proposition~4.3(1)]{CDaHGV}).  Thus $\pi\inv(H_i)$ is an open
  subgroupoid of $\Sigma$ which is subhomogeneous and has compact
  isotropy subgroups by \Cref{inverse image of groupoid with finite
    dad}. We identify $\Cst_\red(\pi\inv(H_i))$ with a
  $\Cst$-subalgebra $B_i$ of $\Cst_\red(\Sigma)$. Since $\pi\inv(H_i)$
  is amenable,
  \Cref{nuclear-dim-for-subhomog-abelian-or-compact-isotropy} gives
  \[\dim_\nuc (\Cst_\red(\pi\inv(H_i)))\leq\dim(U_i) \leq \dim (G\z).\]

  Let $V_\red\colon C(\Sigma\z)\to M(\cs_\red(\Sigma))$ be
  (restriction of) the map from \Cref{rem V_r}.  The calculations that
  verify the hypotheses of \Cref{abstraction-nonunital} with respect
  to $\epsilon$, $\Ff$, $B_i$ and
  $V_\red(h_i)\in V_\red(C(G\z))\subset M(\Cst_\red(\Sigma))$ are the
  same as those in the proof of \Cref{main thm}.  Since the $B_i$ are
  separable and subhomogeneous, \Cref{abstraction-nonunital} gives
  \[\dim\plusone_\nuc(\Cst_\red(\Sigma))\leq
    \dad\plusone(G)\dim\plusone(G\z)
  \]
  when $\Sigma\z$ is compact.

  Next, suppose that $\Sigma\z=G\z$ is not compact. Let $\cG$ be the
  Alexandrov groupoid of \cite[Lemma~3.4]{CDaHGV}, and let
  $(\widetilde \Sigma, \tilde\iota, \tilde\pi)$ be the twist over
  $\cG$ of \cite[Lemma~3.6]{CDaHGV}. Then $\cG\z=\widetilde \Sigma\z$
  is compact. We have $\dad(\cG)=\dad(G)$ by
  \cite[Proposition~3.13]{CDaHGV}, and $\dim(\cG\z)=\dim(G\z)$ by
  \cite[Lemma~2.6]{CDaHGV}. Notice that
  $\widetilde \Sigma=\Sigma\cup\{\infty_z:z\in\TT\}$ and
  $\{\infty_1\}$ is a closed invariant subset of $\tilde
  \Sigma\z$. The restriction of $\widetilde \Sigma$ to
  $\widetilde{\Sigma}\z\setminus\{\infty_1\}$ is $\Sigma$. Since the
  full and reduced norms on $C(\TT)$ agree, by
  \cite[Proposition~5.2]{Williams:groupoid} there is an exact sequence
  \[
    \begin{tikzcd}
      0\arrow[r]&\cs_{r}(\Sigma)\arrow[r]& \cs_{r}(\widetilde \Sigma)
      \arrow[r]& C(\TT) \arrow[r]&0.
    \end{tikzcd}
  \]
  Since $\tilde \Sigma$ has compact unit space, by
  \cite[Proposition~2.9]{WinterZacharias:nuclear} and the compact case
  above,
  \[\dim\plusone_\nuc(\Cst_\red(\Sigma))\leq
    \dad\plusone(G)\dim\plusone(G\z).\qedhere
  \]
\end{proof}

\printbibliography
\end{document}

%%% Local Variables:
%%% mode: latex
%%% TeX-master: t
%%% End: